\documentclass[reqno,11pt]{amsart}

\usepackage[all]{xy}
\usepackage{amssymb}
\usepackage{amsgen}
\usepackage{amsmath}
\usepackage{amsthm}
\usepackage{mathrsfs}
\usepackage{cite}
\usepackage{amsfonts}
\usepackage{tikz}

\newcommand{\Cliq}{\mathsf{Cliq}}
\newcommand{\ilim}{\varprojlim}
\renewcommand{\to}{\longrightarrow}

\newcommand{\R}{{\mathrel{\mathscr R}}} 

\newcommand{\pv}[1]{\mathbf {#1}}
\newcommand{\inv}{^{-1}}
\newcommand{\p}{\varphi}

\newcommand{\ov}[1]{\ensuremath{\overline {#1}}}
\newcommand{\til}[1]{\ensuremath{\widetilde {#1}}}

\newcommand{\wh}{\widehat}

\newcommand{\module}[1]{#1\text{-}\mathbf{mod}}
\newcommand{\Hom}{\mathop{\mathrm{Hom}}\nolimits}
\newcommand{\Ext}{\mathop{\mathrm{Ext}}\nolimits}

\newcommand{\nup}[2][B]{#1_{\not\geq #2}}
\newcommand{\sgn}{\mathop{\mathbf s}\nolimits}

\usepackage{xcolor}

\newcommand{\reduce}{\operatorname{reduce}}
\def\AAA{\mathcal A} 
\def\LLL{\mathcal L} 
\def\FFF{\mathcal F} 
\def\SSS{\mathcal S} 
\def\OM{\mathscr X} 

\def\BQ{B^\prec_Q}
\newcommand\quiverpathgeneric[3]{\left(#1\mid #2 \mid#3\right)}
\newcommand\qpathwithverts[3][\alpha]{\quiverpathgeneric{#2_0}{#1_1,\ldots,#1_#3}{#2_#3}}

\def\source{s}
\def\target{t}
\newcommand\quiverpath[2]{\left(#1_1 \cdots #1_#2\right)}
\newcommand\quiverpathsuffix[3]{\left(#1_#2 \cdots #1_#3\right)}

\newtheorem{Thm}{Theorem}[section]
\newtheorem{Prop}[Thm]{Proposition}

\newtheorem{Lemma}[Thm]{Lemma}
{\theoremstyle{definition}
\newtheorem{Def}[Thm]{Definition}}
{\theoremstyle{remark}
\newtheorem{Rmk}[Thm]{Remark}}
\newtheorem{Cor}[Thm]{Corollary}
{\theoremstyle{remark}
\newtheorem{Example}[Thm]{Example}}

{\theoremstyle{remark}
}
{\theoremstyle{remark}
}
{\theoremstyle{remark}
}

\numberwithin{equation}{section}

\title[Combinatorial Topology and the Global Dimension of Algebras]{Combinatorial Topology and the Global Dimension of Algebras Arising in Combinatorics}

\author{Stuart Margolis}
\address[S.~Margolis]{%
    Department of Mathematics\\
    Bar Ilan University\\
    52900 Ramat Gan\\
    Israel}
\email{margolis@math.biu.ac.il}

\author{Franco Saliola}
\address[F.~~Saliola]{
D{\'e}partement de Math{\'e}matiques -- LaCIM\\
Universit{\'e} du Qu{\'e}bec {\`a} Montr{\'e}al\\
C.P. 8888, Succursale Centre-Ville\\
Montr{\'e}al, Qu{\'e}bec  H3C 3P8\\
Canada
}
\email{saliola.franco@uqam.ca}

\author{Benjamin Steinberg}
\address[B.~Steinberg]{%
    Department of Mathematics\\
    City College of New York\\
    Convent Avenue at 138th Street\\
    New York, New York 10031\\
    USA}
\email{bsteinberg@ccny.cuny.edu}

\thanks{The first author wishes to warmly thank the Center for Algorithmic and Interactive Scientific Software, CCNY, CUNY for inviting him to be a Visiting Professor during part of  the
preparation of this paper.  The second and third authors were supported in part by NSERC. Some of this work was done while the third author was at the School of Mathematics and Statistics of Carleton University}
\date{\today}


\begin{document}

\begin{abstract}

In a highly influential paper, Bidigare, Hanlon and Rockmore showed that a
number of popular Markov chains are random walks on the faces of a hyperplane
arrangement. Their analysis of these Markov chains took advantage of the monoid
structure on the set of faces. This theory was later extended by Brown to a
larger class of monoids called left regular bands. In both cases, the
representation theory of these monoids played a prominent role. In particular,
it was used to compute the spectrum of the transition operators of the Markov
chains and to prove diagonalizability of the transition operators.

In this paper, we establish a close connection between algebraic and
combinatorial invariants of a left regular band: we show that certain
homological invariants of the algebra of a left regular band coincide with the
cohomology of order complexes of posets naturally associated to the left
regular band. For instance, we show that the global dimension of these algebras
is bounded above by the Leray number of the associated order complex.
Conversely, we associate to every flag complex a left regular band whose
algebra has global dimension precisely the Leray number of the flag complex.
\end{abstract}

\maketitle

\section{Introduction}

In a highly influential paper~\cite{BHR}, Bidigare, Hanlon and Rockmore showed
that a number of popular Markov chains, including the Tsetlin library and the
riffle shuffle, are random walks on the faces of a hyperplane arrangement (the
braid arrangement for these two examples).
More importantly, they showed that the representation theory of the monoid of
faces, where the monoid structure on the faces of a central hyperplane
arrangement is given by the Tits projections~\cite{Titsappendix}, could be used
to analyze these Markov chains and, in particular, to compute the spectrum of
their transition operators.

Using the topology of arrangements, Brown and Diaconis~\cite{DiaconisBrown1} found resolutions of the simple modules for the face monoid that were later shown by the second author to be the minimal projective resolutions~\cite{Saliolahyperplane}. Brown and Diaconis used these resolutions to prove diagonalizability of the transition operator. Bounds on rates of convergence to stationarity were obtained in~\cite{BHR,DiaconisBrown1}. They observed, moreover, that one can replace the faces of a hyperplane arrangement by the covectors of an oriented matroid~\cite{OrientedMatroids1993} and the theory carries through. We remark that the original version of Brown's book on buildings~\cite{Brown:book1} makes no mention of the face monoid of a hyperplane arrangement, whereas it plays a prominent role in the new edition~\cite{Brown:book2}. Hyperplane face monoids also have a salient position in the work of Aguiar and Mahajan on combinatorial Hopf algebras~\cite{Aguiar,aguiarspecies}.

The representation theory of hyperplane face monoids is closely connected to Solomon's descent algebra~\cite{SolomonDescent}.  Bidigare showed in his thesis~\cite{Bidigarethesis} (see also~\cite{Brown2}) that if $W$ is a finite Coxeter group and $\AAA_W$ is the associated
reflection arrangement, then the descent algebra of $W$ is the algebra of invariants for the action of $W$ on the algebra of the face monoid of $\AAA_W$. This, together with his study of the representation theory of hyperplane face monoids~\cite{Saliolahyperplane}, allowed the second author to compute the quiver of the descent algebra in types A and B~\cite{SaliolaDescent} (see also~\cite{schocker}).

The face monoid of a hyperplane arrangement satisfies the identities $x^2 = x$ and $xyx = xy$. A semigroup satisfying these identities is known in the literature as a left regular band. Brown developed~\cite{Brown1,Brown2} a theory of random walks on finite left regular bands. He gave numerous examples that do not come from hyperplane arrangements, as well as examples of hyperplane walks that could more easily be modeled on simpler left regular bands. For example, Brown considered random walks on bases of matroids. Brown used the representation theory of left regular bands to extend the spectral results of Bidigare, Hanlon and Rockmore~\cite{BHR} and gave an algebraic proof of the diagonalizability of random walks on left regular bands.

Brown's theory has since been used and further developed by numerous authors.
Diaconis highlighted hyperplane face monoid and left regular band walks in his
1998 ICM lecture~\cite{DiaconisICM}.
Bj\"orner used it to develop the theory of random walks on complex hyperplane
arrangements and interval greedoids~\cite{bjorner1,bjorner2}.
Athanasiadis and Diaconis revisited random walks on hyperplane face monoids and left
regular bands in~\cite{DiaconisAthan}.
Chung and Graham considered further left regular band random walks associated
to graphs in~\cite{GrahamLRB}.
Saliola and Thomas proposed a definition of oriented interval greedoids by
generalizing the left regular bands associated to oriented matroids and
antimatroids \cite{SaliolaThomas}.
See also the recent work of Reiner, Saliola and Welker on symmetrized random
walks on hyperplane face monoids~\cite{Randomtorandom}.
Left regular bands have also appeared in Lawvere's work in topos theory
\cite{graphic1,moregraphic}.

Left regular bands have directed quasi-hereditary algebras and hence have acyclic quivers and finite global dimension. The second author computed in~\cite{Saliolahyperplane} the $\Ext$-spaces between simple modules in the case of the algebras of face monoids of hyperplane arrangements using the resolutions of Brown and Diaconis coming from the topology of hyperplane arrangements. Consequently, he computed the global dimension of these algebras. In~\cite{Saliola} the second author computed the projective indecomposable modules for arbitrary left regular band algebras and also the quiver. In this setting, one didn't seem to have any topology available to compute the minimal resolutions and so he was unable to compute $\Ext$-spaces between simple modules.

The paper~\cite{Saliola} also contained an intriguing unpublished result of Ken Brown stating that the algebra of a free left regular band is hereditary. The proof is via a computation of the quiver and amounts to proving that the dimension of the path algebra is the cardinality of the free left regular band. The right Cayley graph of a free left regular band is a tree (after removing loop edges) and this leads one to suspect that there is a topological explanation for the fact that its algebra has global dimension one. This paper arose in part to give a conceptual explanation to this result of Brown.

There seem to be only a handful of results in the finite dimensional algebra literature that use topological techniques to compute homological invariants of algebras. The primary examples seem to be in the setting of incidence algebras where the order complex of the poset plays a key role~\cite{Cibils,Gerstenhaber,Igusa}. A more general setting is considered in~\cite{Bustamente}. In this paper we use topological techniques to compute the $\Ext$-spaces between simple modules of the algebra of a left regular band. In particular, we use order complexes of posets and classifying spaces of small categories (in the sense of Segal~\cite{GSegal}) to achieve this. A fundamental role is played by the celebrated Quillen's Theorem A, which gives a sufficient condition for a functor between categories to induce a homotopy equivalence of classifying spaces. Somewhat surprisingly to us, a combinatorial invariant of simplicial complexes, the Leray number~\cite{Kalai1,Kalai2}, plays an important part in this paper. The Leray number is tied to the Castelnuovo-Mumford regularity of Stanley-Reisner rings. In particular, the paper gives a new, non-commutative interpretation of the regularity of the Stanley-Reisner ring of a flag complex.

Let us give a more technical overview of the paper. We will assume that all left regular bands are finite. Our goal is to study the
algebra $\Bbbk B$ of a left regular band $B$ over a commutative ring with
unit $\Bbbk$. The reader should feel free to assume that $\Bbbk$ is a field if he/she
likes.  The principal goal is to compute $\Ext^n_{\Bbbk B}(\Bbbk_X,\Bbbk_Y)$ for all $n\geq 0$,
where $\Bbbk_X$ and $\Bbbk_Y$ are certain $\Bbbk B$-modules; in the case that $\Bbbk$ is a
field, these are the simple $\Bbbk B$-modules. Our main result identifies these
$\Ext$-spaces with the cohomology of order complexes of posets naturally
associated to the left regular band. This establishes a close connection
between these algebraic invariants and the combinatorics of these order
complexes. For instance, we show that the global dimension of $\Bbbk B$ is bounded
above by the Leray number~\cite{Kalai1,Kalai2} of the associated order complex.  Conversely, we associate to every flag complex $K$ a left regular band whose algebra has global dimension precisely the Leray number of $K$.

The article is outlined as follows.
In Section \ref{s:LRBs} we recall definitions and properties of left regular
bands and related constructions. Section \ref{s:LRBExamples} surveys several
examples of left regular bands. We illustrate how some of the left regular
bands that have appeared in the literature are special cases of classical
semigroup-theoretic constructions. We also introduce some new examples:
\emph{free partially commutative left regular bands}, which are analogues
of trace monoids and right-angled Artin groups~\cite{BestvinaBrady};
\emph{geometric left regular bands}, which include nearly all the left regular bands
that have appeared in the algebraic combinatorics literature;
and the \emph{left regular band of an acyclic quiver} whose semigroup algebra
is the path algebra of the quiver.

Since the proof of our main theorem is rather involved, we decided to discuss
its applications before presenting its proof. So Section \ref{s:applications}
is devoted to applications of the main theorem. We begin with a new description
of the quiver of a left regular band algebra. We show that the algebra's global dimension
is bounded above by the Leray number of the order complex of the left regular band. This
leads to a characterization of the free partially commutative left regular bands
with a hereditary algebra as those constructed from chordal graphs.

The proof of the main theorem is split across two sections. Section
\ref{s:mainresult} reviews the second author's construction of a complete set of orthogonal idempotents, which
are then used to identify the
Sch\"utzenberger representations of the left regular band as projective modules (indecomposable over a field). We construct
projective resolutions of the modules $\Bbbk_X$, which recasts the computation of
$\Ext^n_{\Bbbk B}(\Bbbk_X,\Bbbk_Y)$ into one involving monoid cohomology and classifying
spaces.

Section \ref{s:MonoidCohomology} contains the crux of the proof. Our main tools
are classifying spaces and the cohomology of monoids and small categories.
Although we are mostly interested in monoid cohomology, which is a natural
generalization of group cohomology, we will also need to work with categories
that are not monoids; namely, posets and the semidirect product of a monoid
with a set (also known as the Grothendieck construction, or category of elements).


\section{Left Regular Bands}
\label{s:LRBs}

\subsection{Bands and left regular bands}\label{ss:LRBs1}
A \emph{band} is a monoid $B$ all of whose elements are idempotents.
A particularly important class of bands arising in probability theory and in
algebraic combinatorics is the class of left regular
bands~\cite{Brown1,Brown2,DiaconisBrown1,Aguiar,Saliola,Saliolahyperplane,bjorner1,bjorner2,GrahamLRB}.
\begin{Def}
    \label{DefLRB1}
A \emph{band} is a monoid $B$ satisfying the identity
\begin{gather}
    \label{idempotent}
    x^2=x \text{ for all } x \in B.
\end{gather}
A band is \emph{left regular} if it satisfies the identity
\begin{gather}
    \label{leftregularity}
    xyx = xy \text{ for all } x,y \in B.
\end{gather}
\end{Def}

The class of left regular bands is hence a variety of bands. It is known
(cf.~\cite[Proposition~7.3.2]{qtheor}) to be generated by the band $\{0,+,-\}$ where $0$ is the identity
element and the binary operation $\circ$ is given by
\begin{gather*}
    + \circ + = + \circ - = +
    \quad\text{ and }\quad
    - \circ + = - \circ - = -
\end{gather*}
Important examples of left regular bands arising in combinatorics are real and
complex hyperplane face semigroups, oriented matroids, matroids and interval
greedoids. Other interesting examples will be seen in Section
\ref{s:LRBExamples}.

It is easy to see that a band $B$ is left regular if and only if each left
ideal is two-sided.  Indeed, if $B$ is left regular and $L$ is a left ideal,
then $a\in L$ implies $ab=aba\in L$ for all $b\in B$.  Conversely, if each
two-sided ideal is a left ideal, then $ab\in Ba$ and so $aba=ab$ for all
$a,b\in B$.

\subsection{Support lattice and the support map}
If $B$ is a band, then it was shown by Clifford~\cite{Clifford,CP} that the
principal ideals are closed under intersection and hence form a (meet)
semilattice $\Lambda(B)$ with maximum. More precisely, he proved that $BaB\cap
BbB=BabB$ and hence $\sigma\colon B\to \Lambda(B)$ given by $\sigma(a)=BaB$ is
a monoid homomorphism.  In fact, $\sigma$ is the universal map from $B$ to a
semilattice.

If $B$ is a finite left regular band, then $\Lambda(B)$ is
the set of principal left ideals, which is a lattice under inclusion with
intersection as the meet.
Following the standard convention of lattice theory, we denote by $\wh 1$ the
top of $\Lambda(B)$ (which is $B$, itself) and by $\wh 0$ the bottom (called the \emph{minimal ideal} of $B$).
Brown calls $\Lambda(B)$ the \emph{support lattice} of $B$~\cite{Brown1,Brown2} (actually, he uses the opposite ordering).  The map
$\sigma\colon B\to \Lambda(B)$, above, becomes $\sigma(a)=Ba$ and is called the
\emph{support map}. It is possible to give a definition of left regular bands
in terms of the support map; see e.g.~\cite[Appendix B]{Brown1} for a proof of the
following.
\begin{Prop}
    \label{DefLRB2}
    A finite monoid $M$ is a \emph{left regular band}
    if and only if there exist a lattice $\Lambda$ and a surjection
    $\sigma\colon M \to \Lambda$ satisfying the following
    properties for all $x,y \in M$:
    \begin{gather}
    \sigma(xy) = \sigma(x) \wedge \sigma(y) \\
    xy = x \text{ if and only if }\sigma(y) \geq \sigma(x)
    \end{gather}
    where $\wedge$ denotes the meet operation (greatest lower bound) of the
    lattice $\Lambda$.
\end{Prop}

In semigroup parlance, this is the well known fact that a left regular band is the same thing as a semilattice of left zero semigroups.

\subsection{Green's $\R$-order}
\label{ss:Rorder}
Let $M$ be a monoid. \emph{Green's $\R$-preorder} is defined on $M$ by
$m\leq_\R n$ if $mM\subseteq nM$. The associated equivalence relation is
denoted $\R$ and is one of \emph{Green's relations} on a monoid. See~\cite{CP,Green}.

If $B$ is a band, one has $aB\subseteq bB$ if and only if $ba=a$.
Hence,
\begin{align*}
    a \leq_\R b\ \text{if and only if}\ ba = a.
\end{align*}
In a left regular band, if $ba=a$ and $ab=b$, then $a=aba=ab=b$.
It follows that a left regular band $B$ is partially ordered with respect to $\leq_\R$. (In fact a band is left regular if and only if $\leq_\R$ is a partial order.) We call this
partial order the \emph{$\R$-order} on $B$ and denote it simply by $\leq$.
Note that the support map $\sigma\colon B \to \Lambda(B)$ is order-preserving:
that is, if $a \leq b$, then $\sigma(a) \leq \sigma(b)$.
Figures \ref{figure: free lrb}, \ref{fig:3LinesRorder} and \ref{multtableRorder}
illustrate the $\R$-order on three examples.

The following is a special case of an elementary result of
Rhodes~\cite{resultsonfinite}.

\begin{Lemma}\label{Rhodeselem}
Let $B$ be a left regular band with support map $\sigma\colon B\to \Lambda(B)$.  Then we have the following.
\begin{enumerate}
\item If $b_0<b_1<\cdots<b_n$ in $B$, then $\sigma(b_0)<\sigma(b_1)<\cdots<\sigma(b_n)$.
\item If $X_0<X_1<\cdots<X_n$ is a chain in $\Lambda(B)$, then there is a chain $b_0<b_1<\cdots<b_n$ in $B$ with $\sigma(b_i)=X_i$ for all $0\leq i\leq n$.
\end{enumerate}
\end{Lemma}
\begin{proof}
For the first statement, it suffices to observe that if $a\leq b$ and $\sigma(a)=\sigma(b)$, then $a=ba=b$.  For the second statement, choose $a_i$ with $\sigma(a_i)=X_i$, for $0\leq i\leq n$, and define $b_i = a_na_{n-1}\cdots a_i$, for $0\leq i\leq n$.
\end{proof}

\subsection{Local, induced and interval submonoids}
\label{ss:intervalsubmonoids}
If $X\in \Lambda(B)$, then
\begin{align*}
B_{\geq X} = \{b\in B\mid \sigma(b)\geq X\}
\end{align*}
is a submonoid of $B$.  The set $\nup X = B\setminus B_{\geq X}$ is a prime
ideal of $B$ and all prime ideals of $B$ are obtained in this way.
(Recall that an ideal $P$ is \emph{prime} if $ab \in P$ implies $a\in P$ or $b
\in P$.)

Note that
\[a^\uparrow = \{b\in B\mid b\geq a\}= \{b\in B\mid ba=a\}\] is a submonoid of $B$,
in fact, it is the left
stabilizer of $a$.  Notice that if $a\in B$, then $aB=aBa$ is a left regular
band with identity $a$ and the map $b\mapsto ab$ gives a retraction
$\tau_a\colon B\to aB$.  The monoid $aB$ is called the \emph{local submonoid}
of $B$ at $a$.  One has $\Lambda(aB)=\Lambda(B)_{\leq\sigma(a)}$, the principal
downset of $\Lambda(B)$ generated by $\sigma(a)$.  If $\sigma(a)=X=\sigma(b)$,
then $aB\cong bB$ via the restriction of $\tau_a$ to $bB$ and the restriction of $\tau_b$ to $aB$. The corresponding left regular band
(well defined up to isomorphism) will be denoted $B[X]$ and called the
\emph{induced submonoid on $X$}.

If $X\leq Y$ in $\Lambda(B)$, then $B_{\geq
X}[Y]\cong B[Y]_{\geq X}$ and this left regular band will be denoted $B[X,Y]$.
It has support lattice the interval $[X,Y]$ of $\Lambda(B)$.  Hence we shall
call $B[X,Y]$ the \emph{interval submonoid} of $B$ associated to $[X,Y]$.  Of
course, $B[\wh 0, X]=B[X]$ and $B[X,\wh 1] = B_{\geq X}$.  It will be
convenient to denote by $B[X,Y)$ the ideal of $B[X,Y]$ obtained by removing the
identity.

\section{Examples of Left Regular Bands}
\label{s:LRBExamples}
This section surveys some examples of left regular bands.
We illustrate how several of the examples of left regular bands found in the combinatorics
literature are special cases of certain semigroup-theoretic constructions; and
we introduce a new class of examples: \emph{free partially commutative left
regular bands}.

\subsection{Free left regular bands}
The \emph{free left regular band} on a set $A$ is denoted $F(A)$ and the free left
regular band on $n$-generators will be written $F_n$.  The word problem for
$F(A)$ is quite elegant.  Let $A^*$ denote the free monoid on $A$.  One can
view $F(A)$ as consisting of all injective words over $A$, i.e., those elements of $A^*$ with no repeated letters.
Multiplication is given by concatenation followed by removal of repetitions
(reading from left to right).  In particular, if $A$ is finite, then so is
$F(A)$.  Hence any finitely generated left regular band is finite (actually any
finitely generated band is finite).  The support lattice of $F(A)$ can be identified with the power set $P(A)$ with the operation of union.  The support map $\sigma$ takes an injective word to its content (or alphabet).  If $C\subsetneq B\subseteq A$, then $F(A)[B,C]\cong F(B\setminus C)$.

\begin{figure}
\begin{gather*}
\xymatrix@l@C=0pt@R=1em{
cba \ar@{-}[d] && cab \ar@{-}[d] && bca \ar@{-}[d] && bac \ar@{-}[d] & & acb \ar@{-}[d] && abc \ar@{-}[d] \\
 cb \ar@{-}[dr] &&  ca \ar@{-}[dl] &&  bc \ar@{-}[dr] &&  ba \ar@{-}[dl] & &  ac \ar@{-}[dr] &&  ab \ar@{-}[dl] \\
    &c \ar@{-}[drrrr] &    &&    &b\ar@{-}[d]&     & &    &a\ar@{-}[dllll]& \\
    &&     &&    &1&
}
\hspace{3em}
\xymatrix@l@C=1em@R=1em{
& \{a,b,c\} \ar@{-}[dl] \ar@{-}[rd] \ar@{-}[d] \\
\{b,c\} \ar@{-}[d]\ar@{-}[dr] & \{a,c\} \ar@{-}[dl]\ar@{-}[dr] & \{a,b\} \ar@{-}[d]\ar@{-}[dl] \\
\{c\} \ar@{-}[dr]  & \{b\} \ar@{-}[d]  & \{a\} \ar@{-}[dl] \\
 & \emptyset
}
\end{gather*}
\caption{The $\R$-order and the support lattice of $F(\{a,b,c\})$.}
\label{figure: free lrb}
\end{figure}
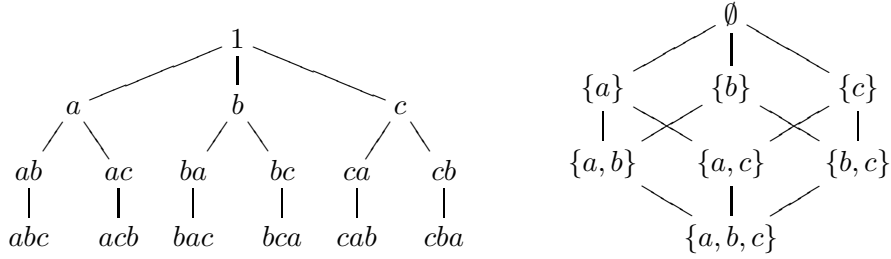

\subsection{Hyperplane face monoids and oriented matroids}
\label{ss:hyperplanefacemonoids}
The set of faces of a hyperplane arrangement, and more generally the set of
covectors of an oriented matroid, is endowed with a natural associative product
providing an important source of examples of left regular bands. These turn out
to be submonoids of $\{0,+,-\}^n$, where $\{0,+,-\}$ is as defined in Section~\ref{ss:LRBs1}.

We recall the construction and properties of these left regular bands referring
the reader to \cite[Appendix~A]{Brown1} for details.
A \emph{central hyperplane arrangement} in $V = \mathbb R^n$ is a finite
collection $\AAA$ of hyperplanes of $V$ passing through the origin. For each
hyperplane $H \in \AAA$, fix a labelling $H^+$ and $H^-$ of the two open half
spaces of $V$ determined by $H$; the choice of labels $H^+$ and $H^-$ is
arbitrary, but fixed throughout. For convenience, let $H^0 = H$.

A \emph{face} $x$ of $\AAA$ is a non-empty intersection of the form
$
x = \bigcap_{H \in \AAA} H^{\epsilon_H}
$
with $\epsilon_H \in \{0,+,-\}$. Consequently, for every hyperplane $H \in
\AAA$, a face of $\AAA$ is contained in exactly one of $H^+$, $H^-$ or $H^0$.
If $y$ is a face of $\AAA$ and $H \in \AAA$, let $\varepsilon_H(y) \in
\{0,+,-\}$ be such that $y \subseteq H^{\varepsilon_H(y)}$. The sequence
$\varepsilon(y) = (\varepsilon_H(y))_{H \in \AAA}$ is called the \emph{sign
sequence} of $y$ and it completely determines $y$.


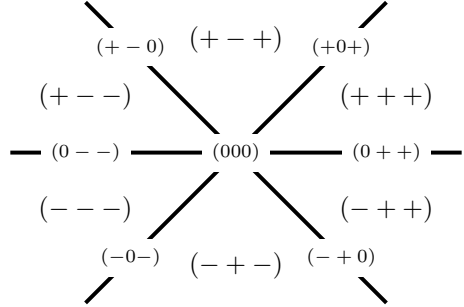
\begin{figure}[htb]
\centering
\begin{tikzpicture}
    \draw[line width=1.5pt] (-2,-2) -- (2,2);
    \draw[line width=1.5pt] (2,-2) -- (-2,2);
    \draw[line width=1.5pt] (-3,0) -- (3,0);
    \draw (0,0)       node[fill=white,font=\tiny] {$(000)$};
    \draw (2,0.75)    node[font=\small]           {$(+++)$};
    \draw (2,0)       node[fill=white,font=\tiny] {$(0++)$};
    \draw (2,-0.75)   node[font=\small]           {$(-++)$};
    \draw (1.4,-1.4)  node[fill=white,font=\tiny] {$(-+0)$};
    \draw (0,-1.5)    node[font=\small]           {$(-+-)$};
    \draw (-1.4,-1.4) node[fill=white,font=\tiny] {$(-0-)$};
    \draw (-2,-0.75)  node[font=\small]           {$(---)$};
    \draw (-2,0)      node[fill=white,font=\tiny] {$(0--)$};
    \draw (-2,0.75)   node[font=\small]           {$(+--)$};
    \draw (-1.4,1.4)  node[fill=white,font=\tiny] {$(+-0)$};
    \draw (0,1.5)     node[font=\small]           {$(+-+)$};
    \draw (1.4,1.4)   node[fill=white,font=\tiny] {$(+0+)$};
\end{tikzpicture}
\caption{The sign sequences of the faces of the hyperplane arrangement in
$\mathbb R^2$ consisting of three distinct lines.}
\label{fig:3Lines}
\end{figure}

\begin{figure}[htb]
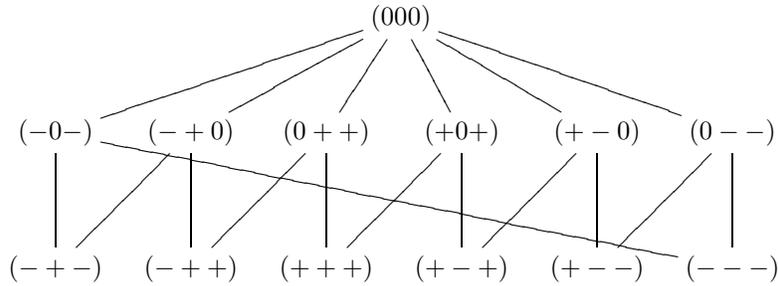

\small
\begin{gather*}
\xygraph{
!{0;/r18mm/:,p+/u18mm/::}
!~:{@{-}}
{(---)}(
    :[u]{(0--)}
    :[dl]{(+--)}
    :[u]{(+-0)}
    :[dl]{(+-+)}
    :[u]{(+0+)}
    :[dl]{(+++)}
    :[u]{(0++)}
    :[dl]{(-++)}
    :[u]{(-+0)}
    :[dl]{(-+-)}
    :[u]{(-0-)}
    :"(---)"),
"{(-0-)}":[rrru(.85)]{(000)}(:"(-+0)",:"{(0++)}",:"(+0+)",:"(+-0)",:"(0--)")
 )}
\end{gather*}
\caption{The $\R$-order on the face monoid of Figure \ref{fig:3Lines}.}
\label{fig:3LinesRorder}
\end{figure}

Let $\FFF$ denote the set of faces of $\AAA$. The image of $\varepsilon$
identifies $\FFF$ with a submonoid of $\{0,+,-\}^{\AAA}$ and so we
obtain a monoid structure on $\FFF$ by defining the product of $x, y \in \FFF$
to be the face with sign sequence $\varepsilon(x) \circ \varepsilon(y)$.
In other words, $xy$ is defined by the property that $xy$ lies: on the same
side of $H$ as $x$ if $x \not\subseteq H$; on the same side of $H$ as $y$ if $x
\subseteq H$, but $y \not\subseteq H$; and inside $H$ if $x,y \subseteq H$.
This product admits an alternative geometric description: $xy$ is the unique
face---possibly $x$ itself---containing the point obtained by moving a small
positive distance along a straight line from a point in $x$ toward a point in
$y$.


The left regular band $\FFF$ is called the \emph{face monoid} of $\AAA$.
The lattice $\Lambda(\FFF)$ of principal left ideals of $\FFF$ is isomorphic to
the \emph{intersection lattice} $\LLL$ of $\AAA$; it is the set of subspaces of
$V$ that can be expressed as an intersection of hyperplanes from $\AAA$ ordered
by reverse inclusion. Under this isomorphism, the universal map $\sigma:
\FFF \to \Lambda(\FFF)$ corresponds to the map from $\FFF$ to $\LLL$ that sends
a face $x$ to the smallest subspace $\bigcap_{\{H \in \AAA\mid x \subseteq H\}} H
\in \LLL$ that contains $x$. Observe that for $X < Y$ in $\LLL$,  $\FFF[X,Y]$ is the face monoid of the hyperplane arrangement in $X$
obtained by intersecting $X$ with the hyperplanes $H \in \AAA$ containing $Y$
but not $X$.

\subsubsection{Oriented matroids}
An \emph{oriented matroid} $\OM$ is an abstraction of the properties enjoyed by a
configuration of vectors in a vector space, or what amounts to the same thing
by working instead with their orthogonal complements, a hyperplane arrangement.
They are also submonoids of the left regular band $\{0, +, -\}^n$
\cite[\S4.1]{OrientedMatroids1993}, but not all submonoids of $\{0, +, -\}^n$
are oriented matroids (\textit{cf}. \S\ref{ss:geometricLRBs}). Much of the
monoid structure of $\OM$, as well as the structure of its monoid algebra,
parallels the theory for the face monoid of a hyperplane arrangement. See
\cite[\S6]{DiaconisBrown1} and \cite[\S11]{Saliolahyperplane} for details.

\subsection{Complex hyperplane arrangements}
In this section we describe a left regular band associated to a complex hyperplane arrangement following~\cite{bjorner2} and~\cite{BjornerZiegler1992}.  All unproved assertions can be found in these references.

Define a left regular band structure on $\SSS=\{0,+,-,i,j\}$ via the multiplication table in Figure~\ref{multtableRorder}.  The Hasse diagram of the $\mathscr R$-order of $\SSS$ is also depicted in Figure~\ref{multtableRorder}.
\begin{figure}[htbp]
\centering
\begin{gather*}
\begin{array}{c}
\begin{array}{c| c c c c c}
 & 0 & + & - & i & j\\ \hline
0& 0& + &- &i &j \\
+ &+&+ & + &i & j\\
- & - &- & -  & i & j\\
i & i & i &i & i & i\\
j & j & j & j & j & j
\end{array}
\end{array}
\hspace{5em}
\begin{array}{c}
\xymatrix{ & 0 & \\
                  +\ar@{-}[ru]&  & -\ar@{-}[lu]\\
                  i\ar@{-}[u]\ar@{-}[rru] &  & j\ar@{-}[u]\ar@{-}[llu]}
\end{array}
\end{gather*}
\caption{The multiplication table and $\R$-order of $\SSS$.}
\label{multtableRorder}
\end{figure}

Define a function $\sgn\colon \mathbb C\to \SSS$ by
\[\sgn(x+iy) = \begin{cases}i, & \text{if}\ y>0,\\ j, & \text{if}\ y <0,\\ +, & \text{if}\ y=0, x>0,\\ -, & \text{if}\ y=0, x<0, \\ 0, & \text{if}\ x=0=y. \end{cases}\]

A \emph{complex hyperplane arrangement} is a set  $\mathcal A=\{H_1,\ldots, H_n\}$ where $H_i$ is the zero set of a complex linear form $f_i$ on $\mathbb C^d$, for $1\leq i\leq n$.  We always assume that $H_1\cap\cdots \cap H_n=\{0\}$.  The position of a point $z\in \mathbb C^d$ relative to $\mathcal A$ can be described by the map $\tau\colon \mathbb C^d\to \SSS^n$ given by \[\tau(z_1,\ldots,z_d)= (\sgn(f_1(z_1)),\ldots, \sgn(f_d(z_d))).\] The image $\FFF=\tau(\mathbb C^d)$ is a submonoid of $\SSS^n$.  Moreover, for each $F\in \FFF$, one has that $\tau^{-1}(F)$ is a relative-open convex cone.  The intersections $\tau^{-1}(F)\cap S^{2d-1}$ are the open cells of a regular CW-decomposition of $S^{2d-1}$ and the face poset of this decomposition is the opposite of the $\mathscr R$-order on $\FFF$ (where the identity corresponds to the empty face).  See~\cite[Theorem~2.5]{BjornerZiegler1992} for details.  For this reason, elements of $\FFF$ will be called \emph{faces} and we will call $\FFF$ the \emph{face monoid} of $\mathcal A$.  The minimal ideal of $\FFF$ consists of all elements of $\FFF\cap \{i,j\}^n$, cf.~\cite[Proposition~3.1]{bjorner2}.  It is shown in~\cite[Theorem~3.5]{BjornerZiegler1992} that the $\mathscr R$-order on the ideal  $\FFF\cap (\SSS\setminus \{0\})^n$ of $\FFF$ is the face poset of a regular CW complex that is homotopy equivalent to the complement $\mathbb C^d\setminus (H_1\cup\cdots \cup H_n)$.

The \emph{augmented intersection lattice} $\LLL$ of $\mathcal A$ is the collection of all intersections of elements of
\[\mathcal A_{\mathrm{aug}} = \{H_1,\ldots, H_n, H_1^{\mathbb R},\ldots, H_n^{\mathbb R}\}\] ordered by reverse inclusion.  Here $H_i^{\mathbb R} = \{z\in \mathbb C^d\mid \sgn(f_i(z))\in \{0,+,-\}\}$, which is a real hyperplane defined by $\Im (f_i(z))=0$.  One has that $\LLL$ is the support lattice of $\FFF$ and the support map takes $F\in \FFF$ to the intersection of all elements of $\mathcal A_{\mathrm{aug}}$ containing $\tau\inv(F)$~\cite[Proposition~3.3]{bjorner2}.  The lattice $\LLL$ is a semimodular lattice of length $2d$~\cite[Proposition~3.2]{bjorner2}.   More generally, if $X<Y$ in $\LLL$ then the length of the longest chain from $X$ to $Y$ in $\LLL$ is $\dim X-\dim Y$.

\subsection{The Karnofsky-Rhodes expansion}
\label{ss:KarnofskyRhodesExpansion}
If $L$ is a lattice generated (under meet) by a finite set $A$, then there is a
universal $A$-generated left regular band with support lattice $L$, known as
the Karnofsky-Rhodes expansion of $L$.  Let us first describe the construction
for monoids in general.

Let $M$ be a monoid with generating set $A$; we do not assume $A\subseteq M$
although we treat it this way notationally.  If $w\in A^*$, then $[w]_M$ will
denote the image of $w$ under the canonical projection $A^*\to M$. Let
$\Gamma_A(M)$ be the \emph{right Cayley graph} of $M$ with respect to the
generators $A$;  so $\Gamma_A(M)$ is the digraph (or quiver, if you like) with
vertex set $M$ and edge set $M\times A$ where the edge $(m,a)$ goes from $m$ to
$ma$.  We usually think of this edge as being labeled by $a$ and draw it
\begin{align*}
    \xymatrix{m\ar[r]^a & ma.}
\end{align*}
Given any vertex $m\in \Gamma_A(M)$ and word $w\in A^*$, there is a unique path
labeled by $w$ with initial vertex $m$ (the terminal vertex will be $m[w]_M$).
We call this the path \emph{read} by $w$ from $m$.
Let us say that an edge $e$ of $\Gamma_A(M)$ is a \emph{transition edge}, if
its initial and terminal vertices are in different strongly connected components of
$\Gamma_A(M)$.

Define an equivalence relation on $A^*$ by putting $u\equiv v$ if and only if:
\begin{itemize}
\item $[u]_M =[v]_M$;
\item the sets of transition edges visited by the paths read from $1$ in
    $\Gamma_A(M)$ by $u$ and $v$ coincide.
\end{itemize}
It is known that $\equiv$ is a congruence~\cite{Elston}; clearly it is
contained in the kernel congruence of the projection $A^*\to M$.  The
\emph{Karnofsky-Rhodes expansion} of $M$ with respect to generators $A$ is
given by $\wh M_A=A^*/{\equiv}$.  This construction is an endofunctor of the
category of $A$-generated monoids (with morphisms preserving generators).
Moreover, the collection of canonical projections $\eta_M\colon \wh M_A\to M$
constitute a natural transformation to the identity functor.

Suppose now that $L$ is an $A$-generated meet semilattice with identity.  Since
each strong component of the Cayley graph of $L$ has a unique vertex,
the word problem for $\wh L_A$ is much simpler.  Let
$w=a_1\cdots a_n$ be in $A^*$ with the $a_i\in A$.  We say that $a_i$ is a
\emph{transition} of $w$ if $[a_1\cdots a_{i-1}]_L>[a_1\cdots a_i]_L$.  The
empty string has no transitions.  Notice that $a_1$ is a transition if and only
if $[a_1]_L\neq 1$.  The transitions of $w$ are exactly the labels of the transition
edges visited in $\Gamma_A(L)$ by the path read from $1$ by $w$.

We say that $w$ is \emph{reduced} if either it is empty, or all its letters are
transitions.  In other words, $w=a_1\cdots a_n$ is reduced if and only if
\begin{align*}
    1>[a_1]_L>[a_1a_2]_L>\cdots>[a_1\cdots a_n]_L.
\end{align*}
Define the reduction $\reduce(w)$ to be the word obtained from $w$ by
erasing all its letters that are not transitions.  Notice that $w$ is reduced
if and only if $w=\reduce(w)$.  It is easy to see from the definition of
the Karnofsky-Rhodes expansion that $w$ and $\reduce(w)$ represent the
same element of $\wh L_A$ and that distinct reduced words represent distinct
elements of $\wh L_A$.  Thus $\wh L_A$ can be viewed as the set of reduced
words with product $vw=\reduce(vw)$.  It is routine to verify that $\wh L_A$ is
a left regular band and $\eta_L\colon \wh L_A\to L$ is the support map.  It
follows from the universal property of the Karnofsky-Rhodes expansion given by
Elston~\cite{Elston} that if $B$ is any $A$-generated left regular band with
support lattice $L$ (with the support map $\sigma$ the identity on $A$), then
there is a unique surjective homomorphism $\p\colon \wh L_A\to B$ such that the
diagram
\begin{align*}
    \xymatrix{{\wh L_A}\ar[rr]^{\p}\ar[rd]_{\eta_L} && B\ar[ld]^\sigma\\ & L &}
\end{align*}
commutes.

\begin{Rmk}
Ken Brown considers reduced words in~\cite{Brown1} as part of his proof of the
diagonalizability of random walks on left regular bands.  His proof essentially
boils down to showing that a random walk on a Karnofsky-Rhodes expansion of a
semilattice is diagonalizable and then deducing the result for arbitrary left
regular bands from this case.
\end{Rmk}

We now consider some examples.  Nearly all of these examples can be found in
Brown's paper~\cite{Brown1}.

\begin{Example}[Free left regular band]
Consider the free semilattice on a finite set $A$; it is the power set $P(A)$
ordered by reverse inclusion (and so the operation is union).  Then a letter
$a_i$ of $w=a_1\cdots a_n$ is a transition if and only if $a_i\notin
\{a_1,\ldots,a_{i-1}\}$.  Thus the reduced words are precisely the injective words.  If $w$ is an arbitrary word, then $\reduce(w)$ is
obtained from $w$ by removing repeated letters (reading from left to right).
Thus $\wh{P(A)}_A=F(A)$.  One can also easily deduce this from the universal property.
\end{Example}

\begin{Example}
The following example is from~\cite[Section 5.1]{Brown1}.  Let $\ov {F}_n$ be
the quotient of $F_n$ that identifies an injective word $w$ of length
$n-1$ with the unique injective word of length $n$ with $w$ as a prefix.
Let $L(n)$ be the quotient of the free semilattice on $n$-generators that
identifies all subsets of size $n-1$ with the subset of size $n$.  Then
$\ov{F}_n=\wh {L(n)}_A$ where $A=\{1,\ldots,n\}$.
\end{Example}

\begin{Example}[$q$-analogue]
Next we consider the example from~\cite[Section~1.4]{Brown1}.  Let $q$ be a
prime power and let $\mathbb F_q$ denote the field of $q$ elements.  Define
$F_{n,q}$ to consist of all ordered linearly independent subsets $(x_1,\ldots,
x_s)$ of $\mathbb F_q^n$ with product
\begin{align*}
    (x_1,\ldots,x_s)(y_1,\ldots,y_t) = (x_1,\ldots,x_s,y_1,\ldots, y_t)^{\wedge}
\end{align*}
where $\wedge$ means delete any vector that is linearly dependent on the
preceding vectors.  Set $L(n,q)$ to be the lattice of subspaces of $\mathbb
F_q^n$ ordered by reverse inclusion and put $A=\mathbb F_q^n\setminus \{0\}$.
Define $\sigma\colon A\to L(n,q)$ by sending a vector $v\neq 0$ to the
one-dimensional subspace it spans.  Then it is easy to see that
$\wh{L(n,q)}_A=F_{n,q}$.
\end{Example}

\begin{Example}[Matroids]
\label{ex:matroids}
The following example is from~\cite[Section~6.2]{Brown1}.
Let $\mathcal M$ be a matroid with underlying set $E$; see~\cite{Oxley} for the
basic definitions.
The associated monoid $M$ consists of all ordered independent subsets
$(x_1,\ldots, x_s)$ of $E$.  The product is given by
\begin{align*}
    (x_1,\ldots,x_s)(y_1,\ldots,y_t) = (x_1,\ldots,x_s,y_1,\ldots, y_t)^{\wedge}
\end{align*}
where $\wedge$ means delete any element that depends on earlier elements.  Let
$A$ be the set of non-loops of $E$ and let $L$ be the lattice of flats of
$\mathcal M$ ordered by reverse inclusion.  Let $\sigma\colon A\to L$ be given
by sending $a\in A$ to its closure.  Then it is not hard to check that $M=\wh
L_A$.
\end{Example}

\begin{Example}[Interval greedoids]
\label{ex:intervalgreedoids}
Let $L$ be a semimodular lattice and let $E$ be the set of join-irreducible
elements of $L$ and make $L$ a monoid via the join operation.  Then the dual
lattice $L^\vee$ is a meet semilattice.  Bj\"orner associates an interval
greedoid and a left regular band to the pair $(L,E)$~\cite{bjorner2}.  It is
easy to see that his left regular band is the submonoid of $\wh {L^\vee}_E$
consisting of those reduced words whose associated chain consists of covers in
the order.
\end{Example}

Notice that if $L$ is an $A$-generated lattice, $X\in L$ and $A_{\geq X}$
denotes the set of elements of $A$ which are above $X$, then $(\wh L_A)_{\geq
X}=(\wh {L_{\geq X}})_{A_{\geq X}}$.

\subsection{The Rhodes expansion}
\label{ss:RhodesExpansion}
Next we consider the Rhodes expansion of a lattice.  The notion is defined more
generally for monoids, cf.~\cite{TilsonXII}. Let $L$ be a finite lattice with
top $\wh 1$, which we view as a monoid via its meet.  If $X\subseteq L$ and
$e\in L$, then put $eX=\{ex\mid x\in X\}$.  If $X\subseteq L$ is a chain, then
$\min X$ denotes the minimum element of $X$.  The (monoidal right) \emph{Rhodes
expansion} of $L$ is the set $\wh L$ of all chains of $L$ containing $\wh 1$
with the multiplication
\begin{align*}
    X\cdot Y = X\cup (\min X)Y.
\end{align*}
More explicitly, the product is given by
\begin{alignat*}2
(\wh 1>x_1>\cdots >x_n)&\cdot (\wh 1>y_1>\cdots>y_m) = \\ &\reduce(\wh 1>x_1>\cdots >x_n\geq x_ny_1\geq \cdots \geq x_ny_m)
\end{alignat*}
where `$\reduce$' means to remove repetitions.  The identity is the chain
consisting only of $\wh 1$.

It is not hard to see that the monoid discussed in~\cite[Section 5.2]{Brown1}
and the monoid $\ov S$ associated to a matroid in~\cite[Section 6.2]{Brown1}
are submonoids of Rhodes expansions of lattices.

\subsection{Free partially commutative left regular bands}
If $\Gamma=(V,E)$ is a simple graph, then the \emph{free partially commutative left
regular band} associated to $\Gamma$ is the left regular band $B(\Gamma)$ with
presentation
\begin{align*}
    \big\langle V \mid xy=yx \text{ for all } (x,y)\in E\big\rangle.
\end{align*}
This is the left regular band analogue of free partially commutative monoids
(also called \emph{trace monoids} or \emph{graph
monoids}~\cite{tracebook,CartierFoata}) and of free partially commutative
groups (also called \emph{right-angled Artin groups} or \emph{graph
groups}~\cite{BestvinaBrady}).  For example, if $\Gamma$ is a complete graph, then
$B(\Gamma)$ is the free semilattice on the vertex set of $\Gamma$, whereas if $\Gamma$ has
no edges, then $B(\Gamma)$ is the free left regular band on the vertex set.

The Tsetlin library Markov chain can be modelled as a hyperplane random walk,
but is most naturally a random walk on the free left regular band, see~\cite{Brown1,Brown2}.  Similarly,
the random walk on acyclic orientations of a graph considered by Athanasiadis
and Diaconis as a function of a hyperplane walk~\cite{DiaconisAthan} is most
naturally a random walk on a free partially commutative left regular band.

Let us first solve the word problem for $B(\Gamma)$.  We need to determine when
two elements of $F(V)$ yield the same element of $B(\Gamma)$.  It turns out
two injective words represent the same element of $\Gamma$ if and only if they represent the
same element of the trace monoid associated to $\Gamma$.  We will not assume, however, prior
knowledge of trace theory.

It will be convenient to denote by $\ov \Gamma$ the complementary graph of
$\Gamma=(V,E)$.  If $W\subseteq V$, then $\Gamma[W]$ denotes the induced
subgraph of $\Gamma$ with vertex set $W$ and $\ov \Gamma[W]$ the induced
subgraph of $\ov \Gamma$ with vertex set $W$.  Note that $\ov {\Gamma[W]} = \ov
\Gamma[W]$.

If $w\in F(V)$ with support $\sigma(w)\subseteq V$, then define an acyclic orientation
$\mathcal O(w)$ of $\ov \Gamma[\sigma(w)]$ by directing the edge $(x,y)$ from
$x$ to $y$ if $x$ comes before $y$ in $w$.  The following theorem is inspired by~\cite[Theorem~2.36]{tracebook}.

\begin{Thm}\label{wordproblem}
Let $\Gamma=(V,E)$ be a graph. Two elements $v,w\in F(V)$ are equal in
$B(\Gamma)$ if and only if $\sigma(v)=\sigma(w)$ and $\mathcal O(v)=\mathcal
O(w)$.   Moreover, if $W\subseteq V$ and $\mathcal O$ is an acyclic orientation
of $\ov\Gamma[W]$, then $\mathcal O=\mathcal O(w)$ if and only if $\sigma(w)=W$
and $w$ is a topological sorting of the directed graph $(\ov\Gamma[W],\mathcal
O)$.
\end{Thm}
\begin{proof}
Suppose first $\sigma(v)=\sigma(w)$ and $\mathcal O(v)=\mathcal O(w)$, call
this orientation $\mathcal O$.  By construction, it follows that $v$ and $w$
are topological sortings of $(\ov \Gamma[W],\mathcal O)$.  But it is well known
that any two topological sortings of an acyclic digraph can be obtained from
each other by repeatedly transposing consecutive vertices which are not
connected by an edge, cf.~\cite[Lemma~2.3.5]{tracebook}.  Thus the defining relations of
$B(\Gamma)$ let us transform $v$ to $w$ (since vertices not connected by an
edge of $\ov \Gamma$ commute in $B(\Gamma)$).  We conclude that $v$ and $w$ are
equal in $B(\Gamma)$.

For the converse, first note that the support map $\sigma\colon F(V)\to P(V)$
factors through $B(\Gamma)$ and so if $v$ and $w$ are equal in $B(\Gamma)$, then
$\sigma(v)=\sigma(w)$; call this common support $W$.  To verify that $\mathcal
O(v)=\mathcal O(w)$, it suffices to show that if $(x,y)$ is an edge of $\ov
\Gamma[W]$, then $x$ and $y$ appear in the same order in both $v$ and $w$.  Define a
mapping $\tau \colon V\to \{0,+,-\}$ by
\begin{align*}
    \tau(z) =
    \begin{cases}
        + & \text{if}\ z=x\\
        - & \text{if}\ z=y\\
        0 & \text{else.}
    \end{cases}
\end{align*}
Then since $(x,y)$ is not an edge of $\Gamma$, it follows that if $(s,t)\in E$,
then at least one of $s$ and $t$ maps to $0$.  Thus $\tau$ extends to a
homomorphism $\tau\colon B(\Gamma)\to \{0,+,-\}$ and so $\tau(v)=\tau(w)$.  But
clearly $\tau(v)$ is $+$ if $x$ appears before $y$ and $-$ if $y$ appears
before $x$, and similarly for $w$.  Thus $\mathcal O(v)=\mathcal O(w)$.
\end{proof}

It follows from the theorem that we can identify $B(\Gamma)$ with the set of
pairs $(W,\mathcal O)$ where $W\subseteq V$ and $\mathcal O$ is an acyclic
orientation of $\ov \Gamma[W]$. A vertex $v \in V$ is identified with the
trivial acyclic orientation on the induced subgraph $\ov\Gamma[\{v\}]$, which
contains no arrows. The product is then given by
\begin{align*}
    (W_1,\mathcal O_1)(W_2,\mathcal O_2)= (W_1\cup W_2,\mathcal O)
\end{align*}
where $\mathcal O$ is the orientation satisfying $(x,y)\in \mathcal O$ if
$(x,y)\in \mathcal O_1$, or if $x\in W_1$ and $y\in W_2\setminus W_1$, or if
$x,y\in W_2\setminus W_1$ and $(x,y)\in\mathcal O_2$.  The support lattice is
$P(V)$ ordered by reverse inclusion. The minimal ideal consists of the acyclic orientations of $\ov \Gamma$.

If $v\in V$ and $(V,\mathcal O)$ is an acyclic orientation of $\ov \Gamma$,
then $v\cdot (V,\mathcal O)$ is the orientation of $\ov \Gamma$ that orients
all edges containing $v$ away from $v$ and which otherwise agrees with
$\mathcal O$.  Thus the random walk on the minimal ideal of $B(\Gamma)$ driven
by a probability supported on $V$ is exactly the random walk on acyclic
orientations of $\ov \Gamma$ considered by Athanasiadis and
Diaconis~\cite{DiaconisAthan}.  The computation of the spectrum and the bounds
on the rate of convergence to stationary can be more easily computed using this
monoid.  Actually, as in the case of the Tsetlin library (which corresponds to
when $\Gamma$ has no edges) to obtain the best bounds on the rate of
convergence, one should use the quotient of $B(\Gamma)$ that identifies the
image of a word in $F(V)$ of length $|V|-1$ with the unique word of length
$|V|$ containing that word as a prefix.

If $U\subsetneq W\subseteq V$, then one easily verifies $B(\Gamma)[W,U]\cong
B(\Gamma[W\setminus U])$.  Thus the interval submonoids of $B(\Gamma)$ are
precisely the free partially commutative monoids on induced subgraphs of
$\Gamma$.

\subsection{Geometric and right hereditary left regular bands}
\label{ss:geometricLRBs}
We say that a left regular band $B$ is \emph{geometric} if
$a^\uparrow = \{b\in B\mid b\geq a\}$ is commutative (and hence a lattice under
the order $\leq$ with the meet given by the product) for all $a\in B$.  For
example, $\{0,+,-\}$ is geometric.  The class of geometric left regular bands is closed
under taking direct products and submonoids.  Therefore, the left regular bands
associated to hyperplane arrangements and oriented matroids are geometric,
whence the name.  If $B$ is geometric, then so is any local submonoid and
interval submonoid of $B$.
An example of a non-geometric left regular band is obtained by taking any left
regular band which is not a lattice and adjoining a multiplicative zero.
However, almost all the left regular bands appearing so far in the algebraic
combinatorics literature are geometric.  The main exception is the class of complex hyperplane arrangement face monoids discussed above.

Let us say that $B$ is \emph{right hereditary}, if the Hasse diagram of the
order $\leq$ is a tree. For example, the free left regular band is right
hereditary, as are the Rhodes and Karnofsky-Rhodes expansions of a lattice. In
particular, the left regular bands associated by Brown to matroids (Example
\ref{ex:matroids}) and by Bj\"orner to interval greedoids (Example
\ref{ex:intervalgreedoids}) are right hereditary.

Notice that if $B$ is right
hereditary, then so are its submonoids, as well as its local submonoids and
interval submonoids. Clearly, a left regular band $B$ which is right hereditary is
geometric since each poset of the form $a^{\uparrow}$ is a chain.  The reason
for the terminology ``right hereditary'' is that a left regular band $B$ is
right hereditary if and only if each right ideal of $B$ is projective in the
category of right $B$-sets~\cite{hereditaryactspaper}.  This is an amusing
coincidence of terminology since we shall see later that the algebra of a right
hereditary, left regular band over a field is hereditary in the ring theoretic
sense!

Free partially commutative left regular bands are also geometric.
This follows from the proof of Theorem \ref{wordproblem}, which shows that
$B(\Gamma)$ embeds in $\{0,1\}^V\times \{0,+,-\}^{\ov E}$, where $\ov E$ is the
edge set of $\ov \Gamma$. In fact, a well-known result from trace theory~\cite[Proposition~5.5.1]{tracebook} implies that each subset $a^{\uparrow}$ of $B(\Gamma)$ is a distributive
lattice.

\subsection{Left regular bands associated to acyclic quivers}
\label{ss:quiverLRBs}
Here we construct from any finite acyclic quiver $Q$ a left regular band $\BQ$
whose monoid algebra is isomorphic to the path algebra of $Q$. For quivers and
their path algebras, we adhere to the notation and conventions of \cite{assem}.

Let $Q = (Q_0, Q_1)$ be a finite acyclic quiver with vertex set $Q_0$ and set of
arrows $Q_1$. If $\alpha \in Q_1$, let $\source(\alpha)$ denote its source and
$\target(\alpha)$ denote its target.
A \emph{path} in $Q$ of \emph{length} $l$ from $v_0$ to $v_l$ is a sequence
$\qpathwithverts{v}{l}$, or $\quiverpath{\alpha}{l}$ for brevity, satisfying:
$\alpha_i \in Q_1$ for all $1 \leq i \leq l$;
$\source(\alpha_1) = v_0$;
$\target(\alpha_{i}) = \source(\alpha_{i+1})$ for all $1 \leq i < l$;
and $\target(\alpha_{l}) = v_l$.
If $v \in Q_0$, we denote by $\varepsilon_v$ the \emph{stationary} (or empty) path
$(v\mid\mid v)$ of length $0$ with source and target equal to $v$.

The \emph{path algebra} $\Bbbk Q$ of $Q$ with coefficients in a commutative ring with unit $\Bbbk$
consists of formal $\Bbbk$-linear combinations of paths in $Q$; the product of two
paths is
\begin{align*}
    \quiverpath{\alpha}{r} & \quiverpath{\beta}{s}
    =
    \begin{cases}
        (\alpha_1\cdots\alpha_r\beta_1\cdots\beta_s), & \text{if }
        \target(\alpha_r) = \source(\beta_1), \\
        0, & \text{if } \target(\alpha_r) \neq \source(\beta_1). \\
    \end{cases}
\end{align*}
In particular,
\begin{align}
    \label{vertexproducts}
    \varepsilon_v\varepsilon_u = \delta_{v,u} \varepsilon_v
\end{align}
where $\delta_{u,v} = 1$ if $u=v$ and $\delta_{u,v} = 0$ if $u\neq v$.

\begin{Def}
Fix a total order $\prec$ on $Q_0$ with the property that
$\source(\alpha) \prec \target(\alpha)$ for every arrow
$\alpha$ in $Q_1$.
For a path $\quiverpath{\alpha}{n}$ in $Q$, define
\begin{gather*}
    \ell\quiverpath{\alpha}{n}
    = \sum_{u \succeq \target(\alpha_n)} \varepsilon_u
    + \sum_{i = 1}^n \quiverpathsuffix{\alpha}{i}{n}
\end{gather*}
and let $\BQ$ denote the image of the paths of $Q$ under the map $\ell$:
\begin{align*}
    \BQ = \{ \ell\quiverpath{\alpha}{n} \mid \quiverpath{\alpha}{n}\ \text{is a path of}\ Q\}.
\end{align*}
\end{Def}
That such a total order exists follows from the assumption that $Q$ is acyclic.
Of course, the definition depends on the choice of total order.

\begin{Prop}
    \label{prop:quiverLRBproduct}
    Suppose $\quiverpath{\alpha}{n}$ and $\quiverpath{\beta}{m}$
    are two paths in $Q$.
    \begin{enumerate}
        \item
            If $\target(\alpha_n) \succeq \target(\beta_m)$, then
            \begin{align*}
                \ell\quiverpath{\alpha}{n} \cdot \ell\quiverpath{\beta}{m}
                = \ell\quiverpath{\alpha}{n}.
            \end{align*}
        \item
            If there exists $r$ such that
            $\target(\alpha_n) = \source(\beta_r)$,
            then
            \begin{align*}
                \ell\quiverpath{\alpha}{n} \cdot \ell\quiverpath{\beta}{m}
                = \ell(\alpha_1\cdots\alpha_n \beta_r\cdots\beta_m).
            \end{align*}
        \item
            If there exists $r$ such that
            $\source(\beta_{r-1}) \prec \target(\alpha_n) \prec \source(\beta_r)$,
            then
            \begin{align*}
                \ell\quiverpath{\alpha}{n} \cdot \ell\quiverpath{\beta}{m}
                = \ell\quiverpathsuffix{\beta}{r}{m}.
            \end{align*}
where $\source(\beta_0)$ is understood as smaller than all vertices.
    \end{enumerate}
\end{Prop}

\begin{proof}
The product
$\ell\quiverpath{\alpha}{n} \cdot \ell\quiverpath{\beta}{m}$
expands as
\begin{align*}
    & \sum_{u \succeq \target(\alpha_n)} \sum_{v \succeq \target(\beta_m)} \varepsilon_u \varepsilon_v
    + \sum_{i=1}^n \sum_{v \succeq \target(\beta_m)} \quiverpathsuffix{\alpha}{i}{n} \varepsilon_v \\
    &\quad +
      \sum_{j=1}^m \sum_{u \succeq \target(\alpha_n)} \varepsilon_u \quiverpathsuffix{\beta}{j}{m}
      + \sum_{i=1}^n \sum_{j=1}^m \quiverpathsuffix{\alpha}{i}{n} \quiverpathsuffix{\beta}{j}{m}.
\end{align*}
Since $\varepsilon_u\varepsilon_v = \delta_{u,v}\varepsilon_u$,
it follows that
\begin{align*}
    \sum_{u \succeq \target(\alpha_n)} \sum_{v \succeq \target(\beta_m)} \varepsilon_u \varepsilon_v
    =
    \begin{cases}
        \sum_{u \succeq \target(\alpha_n)} \varepsilon_u,
            & \text{if } \target(\alpha_n) \succeq \target(\beta_m), \\
        \sum_{v \succeq \target(\beta_m)} \varepsilon_v,
            & \text{if } \target(\beta_m) \succeq \target(\alpha_n).
    \end{cases}
\end{align*}
Since
$\quiverpathsuffix{\alpha}{i}{n} \varepsilon_v
= \delta_{v, \target(\alpha_n)} \quiverpathsuffix{\alpha}{i}{n}$,
\begin{align*}
    \sum_{i=1}^n \sum_{v \succeq \target(\beta_m)}
        \quiverpathsuffix{\alpha}{i}{n} \varepsilon_v
    =
    \begin{cases}
        \sum_{i=1}^n \quiverpathsuffix{\alpha}{i}{n},
            & \text{if } \target(\alpha_n) \succeq \target(\beta_m), \\
        0,  & \text{if } \target(\alpha_n) \not\succeq \target(\beta_m).
    \end{cases}
\end{align*}
Since
$\varepsilon_u \quiverpathsuffix{\beta}{j}{m}
= \delta_{u, \source(\beta_j)} \quiverpathsuffix{\beta}{j}{m}$,
\begin{align*}
    \sum_{j=1}^m \sum_{u \succeq \target(\alpha_n)} \varepsilon_u \quiverpathsuffix{\beta}{j}{m}
    =
    \begin{cases}
        0, &
        \text{if } \source(\beta_j) \not\succeq \target(\alpha_n)
        \text{ for all } j, \\
        \sum_{j=r}^m \quiverpathsuffix{\beta}{j}{m},
        &
        \text{if } \source(\beta_{r-1}) \prec \target(\alpha_n) \preceq \source(\beta_r).
    \end{cases}
\end{align*}
Since
$\quiverpathsuffix{\alpha}{i}{n} \quiverpathsuffix{\beta}{j}{m} = 0$
unless $\target(\alpha_n) = \source(\beta_j)$ for some $1 \leq j \leq m$,
\begin{align*}
    & \sum_{i=1}^n \sum_{j=1}^m \quiverpathsuffix{\alpha}{i}{n} \quiverpathsuffix{\beta}{j}{m} \\
    & \qquad =
    \begin{cases}
        0, &
            \text{if } \target(\alpha_n) \neq \source(\beta_j) \text{ for all } j, \\
        \sum_{i=1}^n \left(\alpha_i\cdots\alpha_n \beta_r\cdots\beta_m\right), &
            \text{if } \target(\alpha_n) = \source(\beta_r) \text{ for some } r.
    \end{cases}
\end{align*}
The result now follows by combining these identities.
\end{proof}

\begin{Thm}
$\BQ$ is a left regular band and $\Bbbk\BQ = \Bbbk Q$.
\end{Thm}

\begin{proof}
    We prove $\BQ$ is a left regular band using Proposition \ref{DefLRB2}.
    Let $\Lambda = (Q_0, \prec)^{op}$ denote the lattice with elements $Q_0$
    ordered by the opposite of $\prec$. Define a map $\sigma\colon \BQ \to
    \Lambda$ by $\sigma\quiverpath{\alpha}{n} = \target(\alpha_n)$. By
    Proposition \ref{prop:quiverLRBproduct},
    $\sigma(xy) = \sigma(x)\wedge\sigma(y)$,
    where $\wedge$ denotes the meet operator of $\Lambda$,
    and
    $xy = x$
    iff
    $\sigma(x) \succeq \sigma(y)$.
    Hence, the conditions of Proposition \ref{DefLRB2} are satisfied.  Note that $\BQ$ is a monoid with identity the stationary path at the smallest object of $Q_0$.

It is easy to see that $\BQ$ is a basis for $\Bbbk Q$. Indeed, if we order the basis of paths so that each suffix of a path is larger than the path itself and so that stationary paths of vertices are ordered using $\prec$ in the obvious way, then the map sending a path $\quiverpath{\alpha}{n}$ to $\ell\quiverpath{\alpha}{n}$ has an upper triangular matrix with ones along the diagonal.
\end{proof}

Note that $\BQ$ is almost never a geometric left regular band.

%
%

\subsection{Idempotent inner derivations}

Let $A$ be an associative algebra over a field $\Bbbk$. For an element $a \in A$,
let $\partial_a(x) = xa - ax$ denote the associated inner derivation.
The following was noticed by Lawvere in his work on graphical toposes (see below).

\begin{Prop}[Lawvere~\cite{Lawvere}]
    Let $A$ be an associative algebra over a field $\Bbbk$ of characteristic
    different from $2$. The set of idempotent elements of $A$ for which the
    associated inner derivation is also idempotent
    \[\left\{a \in A \mid a^2 = a\ \text{and}\ \partial_a^2 = \partial_a \right\}\]
    is a left regular band.
\end{Prop}


\subsection{Graphical topos}

Left regular bands have also appeared in topos theory. In \cite{graphic1,moregraphic}, Lawvere
introduced a class of toposes called graphical. A \emph{graphical topos} is a
topos which is generated by objects whose endomorphism monoid is a finite left
regular band. The topos of $B$-sets, where $B$ is a finite left regular band,
is a graphical topos. He used a different name for left regular bands: he
called them \emph{graphical monoids} instead; and he named the identity
$xyx=xy$ the \emph{Sch\"utzenberger-Kimura identity}.  The reason for the terminology `graphical' is that if $B=\{0,+,-\}$, then the category of $B$-sets can be identified with the category of directed graphs with morphisms that are allowed to collapse an edge to a vertex.

\section{Applications of the Main Result}
\label{s:applications}

Since the proof of the main result is rather involved, we defer it to
Sections \ref{s:mainresult} and \ref{s:MonoidCohomology}. This section is
devoted to applications of the main result. We begin by stating the main result. Fix a finite
left regular band $B$ and a commutative ring with unit $\Bbbk$.

\subsection{Statement of the main result}

Let $X\in \Lambda(B)$. Let $\Bbbk_X$ denote the ring $\Bbbk$ viewed as a $\Bbbk B$-module
via the action
\begin{align*}
    b \cdot \alpha =
    \begin{cases}
        \alpha, & \text{if}\ \sigma(b)\geq X,\\
        0, & \text{otherwise},
    \end{cases}
\end{align*}
for all $b \in B$ and $\alpha \in \Bbbk$. These are the simple $\Bbbk B$-modules when
$\Bbbk$ is a field (Section \ref{ss:mainresult}). The main result of this paper is
a computation of $\Ext^n_{\Bbbk B}(\Bbbk_X,\Bbbk_Y)$ for all $n\geq 0$ and all $X,Y \in
\Lambda(B)$. It turns out that this coincides with the cohomology of a
simplicial complex associated to the interval subsemigroup $B[X,Y)$ of $B$
(Section \ref{ss:intervalsubmonoids}). Recall that $B[X,Y)$ is a poset with
respect to Green's $\R$-order: $a\leq b$ if $ba=a$ (see Section
\ref{ss:Rorder}). Let $\Delta(X,Y)$ be its order complex, which is the
simplical complex whose vertex set is $B[X,Y)$ and whose simplices are the
finite chains in $B[X,Y)$.  Our main result is the following theorem.

\begin{Thm}
\label{mainresult}
Let $B$ be a finite left regular band and let $\Bbbk$ be a commutative ring with
unit.  Let $X,Y\in \Lambda(B)$.  Then
\begin{align*}
    \Ext^n_{\Bbbk B}(\Bbbk_X,\Bbbk_Y) =
    \begin{cases}
        \til H^{n-1}(\Delta(X,Y),\Bbbk), & \text{if}\ X<Y,\ n\geq 1,\\
        \Bbbk, &\text{if}\ X=Y,\ n=0,\\
        0,  & \text{otherwise.}
    \end{cases}
\end{align*}
\end{Thm}


\subsection{The quiver of a left regular band}
\label{ss:quiverofaLRB}
Assume for the moment that $\Bbbk$ is a field. Recall that a finite dimensional $\Bbbk$-algebra $A$ is \emph{split basic} if each of its simple modules is one-dimensonal.  It is well known that $\sigma\colon
\Bbbk B\to \Bbbk\Lambda(B)$ is the semisimple quotient and $\Bbbk\Lambda(B)\cong
\Bbbk^{\Lambda(B)}$, \textit{cf}.~\cite{Brown1}.  Thus $\Bbbk B$ is a split basic
$\Bbbk$-algebra. The (Gabriel) \emph{quiver} $Q(A)$ of a unital split basic $\Bbbk$-algebra $A$
is the directed graph with vertices the isomorphism classes of simple
$A$-modules and with the number of edges from $S_1$ to $S_2$ equal to $\dim_{\Bbbk}
\Ext^1_A(S_1,S_2)$.

The second author computed the quiver of a left regular band algebra in~\cite{Saliola}.
We give a new description here, using Theorem~\ref{mainresult}, which is more
conceptual and therefore sometimes easier to apply.

\begin{Thm}\label{quiver}
Let $B$ be a finite left regular band and $\Bbbk$ a field.  Then the quiver $Q(\Bbbk B)$
has vertex set $\Lambda(B)$.  The number of arrows from $X$ to $Y$ is zero
unless $X<Y$, in which case it is one less than the number of connected
components of $\Delta(X,Y)$.
\end{Thm}

It is not hard to see that the number of connected components of the order
complex $\Delta(P)$ of a poset $P$ coincides with the number of equivalence
classes of the equivalence relation on $P$ generated by the partial order, or equivalently, with the number of components of the Hasse diagram of $P$.
From this observation, it is straightforward to verify that our description of
the quiver $Q(\Bbbk B)$ coincides with the one in~\cite{Saliola}.

Let $A$ be a split basic algebra with an acyclic quiver $Q$.
A result of Bongartz~\cite{Bongartz} says that if $S_1,S_2$ are the simple $A$-modules corresponding to vertices $v_1$ and $v_2$ of $Q$, then the number of relations involving paths from $v_1$ to $v_2$ in a minimal quiver presentation of $A$ is the dimension of $\Ext^2_A(S_1,S_2)$.  Thus Theorem~\ref{mainresult} admits the following corollary.

\begin{Cor}
If $X<Y$ in $\Lambda(B)$, then the number of relations involving paths from $X$ to $Y$ in a minimal quiver presentation of $\Bbbk B$ is given by $\dim_{\Bbbk} \til H^1(\Delta(X,Y),\Bbbk)$.
\end{Cor}

\subsection{Global dimension of left regular band algebras}
Let us next apply Theorem~\ref{mainresult} to global dimension.
For a detailed discussion of global dimension the reader is referred
to~\cite{CartanEilenberg,assem,benson}.

The notion of global dimension can be formulated in terms of either left modules or
right modules, but it is well known that these two formulations coincide for a
finite dimensional algebra $A$ over a field $\Bbbk$. Thus, it suffices to define
the notion for left modules.

The \emph{global dimension} $\mathrm{gl.}\dim A$ of a finite dimensional
algebra $A$ over a field $\Bbbk$ is the least $n$ such that $\Ext^{n+1}_A(V,W)=0$
for all finite dimensional $A$-modules $V,W$.
A simple induction on the length of a composition series and application of the
long exact sequence for $\Ext$ shows that $\mathrm{gl.}\dim A$ is the least $n$
such that $\Ext_A^{n+1}(S_1,S_2)=0$ for all simple $A$-modules
$S_1,S_2$, or equivalently the least $n$ such that $\Ext_A^{m}(S_1,S_2)=0$ for all $m>n$ and all simple $A$-modules
$S_1,S_2$~\cite{assem,benson,AuslanderReiten}.

An algebra $A$ has global dimension zero if and only if it is semisimple.
A finite dimensional algebra $A$ is said to be \emph{hereditary} if
$\mathrm{gl.}\dim A\leq 1$.  This is equivalent to the property that each
submodule of a projective $A$-module is projective, as well as to the property
that each left ideal of $A$ is a projective module.  It follows from Gabriel's
theorem that a split basic $\Bbbk$-algebra (such as the algebra of a finite left
regular band) is hereditary if and only if its quiver is acyclic and it is
isomorphic to the path algebra of its
quiver~\cite{assem,benson,AuslanderReiten}.

Let $B$ be a finite left regular band. Nico's results~\cite{Nico1,Nico2} on
global dimension of the algebra of a von Neumann regular semigroup imply that
the global dimension of $\Bbbk B$ is finite and bounded by the length of the longest
chain in $\Lambda(B)$ where the length of a chain $C$ in a poset is $|C|-1$.
This can also be deduced from the theory of quasi-hereditary
algebras~\cite{quasihered} because $\Bbbk B$ is a directed quasi-hereditary algebra
with weight poset the opposite of $\Lambda(B)$~\cite{Putcharep3,rrbg}.  We
easily recover Nico's result, restricted to left regular bands, using
Theorem~\ref{mainresult}.

\begin{Thm}\label{globaldimension}
Let $B$ be a finite left regular band and $\Bbbk$ a field.  Then
\begin{align*}
    \mathrm{gl.}\dim \Bbbk B = \min \{n\mid \til H^n(\Delta(X,Y),\Bbbk)=0,\
                            \text{for all}\ X<Y\in\Lambda(B)\}.
\end{align*}
In particular, one has
\begin{align*}
    \mathrm{gl.}\dim \Bbbk B\leq m
\end{align*}
where $m$ is the length of
the longest chain in $\Lambda(B)$.
\end{Thm}
\begin{proof}
The first statement on global dimension is immediate from
Theorem~\ref{mainresult}.  If $P$ is a finite poset, then the dimension of
$\Delta(P)$ is the length of the longest chain in $P$.   By Lemma~\ref{Rhodeselem} the longest chain in $\Delta(\wh
0,\wh 1)$ has length $m-1$.  Therefore, for $X<Y$, one has $\dim
\Delta(X,Y)\leq \dim \Delta(\wh 0,\wh 1)=m-1$ and so $\til
H^{m}(\Delta(X,Y),\Bbbk)=0$.  Thus $\mathrm{gl.}\dim \Bbbk B\leq m$.
\end{proof}

As an immediate corollary, we obtain a characterization of those algebras $\Bbbk B$
that are hereditary in terms of the order complex of $B$.
\begin{Cor}
    $\Bbbk B$ is hereditary if and only if each connected component of each
    simplicial complex $\Delta(X,Y)$, for $X < Y \in \Lambda(B)$, is acyclic.
\end{Cor}

It is well known that, for a finite dimensional algebra $A$, the projective dimension of a finite dimensional $A$-module $M$ is the least $d$ such that $\Ext_A^{d+1}(M,S)=0$ for all simple $A$-modules $S$. The global dimension $A$ coincides with the maximum projective dimension of a simple $A$-module.  Theorem~\ref{mainresult} thus yields the following refinement of Theorem~\ref{globaldimension}.

\begin{Cor}
\label{pdim}
Let $B$ be a finite left regular band and let $\Bbbk$ be a field.  Let $X\in \Lambda(B)$.  Then the projective dimension of $\Bbbk_X$ is given by
\[\mathrm{proj}.\dim \Bbbk_X = \min \{n\mid \til H^n(\Delta(X,Y),\Bbbk)=0,\
                            \forall Y>X\}.\]
\end{Cor}

\subsection{Leray numbers and an improved upper bound}

We can improve greatly on Nico's upper bound on the global dimension in the case of left regular bands.
First we recall the notion of the Leray number of a simplicial complex. If $K$
is a simplicial complex with vertex set $V$ and $W\subseteq V$, then the
induced subcomplex $K[W]$ is the subcomplex consisting of all simplices whose
vertices belong to $W$.

The \emph{$\Bbbk$-Leray number} of a finite simplicial complex $K$ with vertex set
$V$ is defined by
\begin{align*}
    L_\Bbbk(K) =\min\{d\mid \til H^i(K[W],\Bbbk) =0, \forall i\geq d, \forall W\subseteq V\}.
\end{align*}
See for instance~\cite{Kalai1,Kalai2}.  Originally, interest in Leray numbers came about because of connections with Helly-type theorems~\cite{Wegner,BolandLek,Kalai2}:  the Leray number of a simplicial complex provides an obstruction for realizing the complex as the nerve of a collection of compact convex subsets of $\mathbb R^d$.

Leray numbers also play a role in combinatorial commutative algebra. Let $K$ be a simplicial complex with vertex set $\{x_1,\ldots,x_n\}$. Recall that the \emph{face ideal} $I_K$ of $K$ is the ideal of the polynomial ring $\Bbbk[x_1,\ldots,x_n]$ generated by all square-free monomials $x_{i_1}\cdots x_{i_m}$ with $\{x_{i_1},\ldots, x_{i_m}\}$ not a face of $K$.  The \emph{Stanley-Reisner ring} of $K$ over $\Bbbk$ is $\Bbbk[x_1,\ldots,x_n]/I_K$~\cite{Stanleycommut}.
The $\Bbbk$-Leray number of a simplicial complex $K$ turns out to be the Castelnuovo-Mumford
regularity of the Stanley-Reisner ring of $K$ over $\Bbbk$~\cite{Kalai2,Dochtermann,Woodroofe}. Equivalently, the regularity of the ideal $I_K$ is $L_{\Bbbk}(K)+1$~\cite{Kalai1}.

Theorem~\ref{globaldimension} implies the following upper bound on the global dimension of a left regular band algebra.

\begin{Thm}\label{Leraybound}
Let $B$ be a finite left regular band and $\Bbbk$ a field.  Then $\mathrm{gl.}\dim
\Bbbk B$ is bounded above by the Leray number $L_\Bbbk(\Delta(B))$ of the order complex
of $B$.
\end{Thm}

\subsection{Right hereditary left regular bands have hereditary algebras}

To apply the bound of Theorem \ref{Leraybound}, we need the well-known notion
of the \emph{clique complex} or \emph{flag complex} of a graph.  If
$\Gamma=(V,E)$ is a simple graph, then the clique complex $\Cliq(\Gamma)$ is the
simplicial complex whose vertex set is $V$ and whose simplices are the finite
cliques (subsets of vertices inducing a complete subgraph).  Notice that
$\Gamma$ is the $1$-skeleton of $\Cliq(\Gamma)$ and that $\Cliq(\Gamma)$ is obtained by
`filling in' the $1$-skeleton of every $q$-simplex found in $\Gamma$.   If
$W\subseteq V$, then $\Cliq(\Gamma)[W]=\Cliq(\Gamma[W])$.

As an example, let $P$ be a poset. The \emph{comparability graph} $\Gamma(P)$
is the graph with vertex set $P$ and edge set all pairs $(p,q)$ with $p<q$ or
$q<p$. It is immediate from the definition that $\Cliq(\Gamma(P))=\Delta(P)$.

It is known that $L_\Bbbk(K)=0$ if and only if $K$ is a simplex and $L_\Bbbk(K)\leq 1$
if and only if $K$ is the clique complex of a chordal graph, cf.~\cite{Wegner,BolandLek}.  Recall that a graph $\Gamma$ is \emph{chordal} if it
contains no induced cycle of length greater than or equal to $4$.  Of course,
any induced subgraph of a chordal graph is chordal.  Let us sketch a proof.

\begin{Prop}[Folklore]\label{chordalgraph}
Let $K$ be a finite simplicial complex.  Then $L_\Bbbk(K)\leq 1$ if and only if $K$
is the clique complex of a chordal graph.  Moreover, $L_\Bbbk(K)=0$ if and only if
$K$ is a simplex.
\end{Prop}

\begin{proof}
Since every induced subcomplex of a simplex is a simplex, clearly the Leray
number of a simplex is $0$.

Suppose that $L_\Bbbk(K)\leq 1$. Then $K$ is the clique complex of its $1$-skeleton
$\Gamma$.  Otherwise, there is a clique $W$ in $\Gamma$ which is not a simplex.
Necessarily $|W|\geq 3$.  If we choose $W$ to be of minimal size, then we have
the induced subcomplex $K[W]$ is topologically a sphere of dimension $|W|-2$.
Thus $L_\Bbbk(K)\geq |W|-1\geq 2$.  Thus $K$ is the clique complex of $\Gamma$. If
$K$ is not a simplex, then there is a pair of vertices $\{v,w\}$ which do not
form an edge.  The induced subcomplex $K[\{v,w\}]$ is not connected and so
$L_\Bbbk(K)>0$. Therefore, $L_\Bbbk(K)=0$ implies $K$ is a simplex.
Next suppose that $C_n$ is an induced $n$-cycle in $\Gamma$ with $n\geq 4$.
Then since $\Cliq(C_n)=C_n$ and $\til H^1(C_n,\Bbbk)\cong \Bbbk$, it follows that
$L_\Bbbk(\Cliq(\Gamma))\geq 1$.  Thus $L_\Bbbk(\Cliq(\Gamma))\leq 1$ implies $\Gamma$ is
chordal.

For the converse, it suffices to show that every connected chordal graph
$\Gamma=(V,E)$ has a contractible clique complex, cf.~\cite[Lemma~3.1]{Dochtermann}. This is done by induction on
the number of vertices of $\Gamma$. The proof relies on a classical result of
Dirac~\cite{Dirac} and of Boland and Lekkerkerker \cite{BolandLek} that every chordal graph has a simplicial
vertex $v$, that is, a vertex $v$ whose neighbors form a clique.  By induction
$\Cliq(\Gamma[V\setminus \{v\}])$ is contractible.  Let $X$ be the set of neighbors
of $v$.  Then $X$ is a simplex of $\Cliq(\Gamma[V\setminus \{v\}])$ and $\Cliq(\Gamma)$
is obtained from $\Cliq(\Gamma[V\setminus\{v\}])$ by attaching the simplex $Y=X\cup
\{v\}$ along the facet $X$.  Thus we can collapse $Y$ into $X$, yielding a
simple homotopy equivalence of $\Cliq(\Gamma)$ and $\Cliq(\Gamma[V\setminus \{v\}])$.
This completes the proof.
\end{proof}

From Proposition~\ref{chordalgraph}, it follows that the bound in
Theorem~\ref{globaldimension} is not tight.  Indeed, if $L$ is a finite
lattice, viewed as a left regular band via the meet operation, then
$\mathrm{gl.}\dim \Bbbk L=0$, but unless $L$ is a chain, its order complex is not a
simplex.  However, Theorem~\ref{globaldimension} is tight for right hereditary
left regular bands.

\begin{Thm}\label{firstrightheredproof}
Let $B$ be a finite left regular band and $\Bbbk$ a field.  Suppose that the Hasse
diagram of $B$ is a rooted tree, i.e., $B$ is right hereditary.  Then $\Bbbk B$ is
hereditary.
\end{Thm}
\begin{proof}
It is well known that the comparability graph of a poset whose Hasse diagram is
a rooted tree is chordal.  In fact, it is known that a finite graph has no
induced simple path on four nodes and no induced cycle on four nodes if and
only if it is the comparability graph of a finite poset whose Hasse diagram is
a disjoint union of rooted trees~\cite{Wolk}.  The idea is that any
simple path of length $3$ of the form $a<b<c$ or $a>b>c$ has a chord in the
induced subgraph.  In a disjoint union of rooted trees, there are no simple
paths of the form $a>b<c$.
\end{proof}

Since the free left regular band is right hereditary, this provides a
conceptual proof that the algebra of a free left regular band is
hereditary, a result first proved by K.~Brown using quivers
and a counting argument~\cite[Theorem~13.1]{Saliola}.  We shall momentarily give another, more transparent
proof of Theorem~\ref{firstrightheredproof}, which also makes it simple to
compute the quiver of a right hereditary left regular band.

\subsection{Geometric left regular bands and commutation graphs}
Suppose now that $B$ is a finite geometric left regular band. Since this class
is closed under taking interval submonoids, we can restrict our attention to
$\Delta(\wh 0,\wh 1)$. We provide a simplicial complex homotopy equivalent to
$\Delta(\wh 0,\wh 1)$, using the following special case of Rota's cross-cut theorem, see the survey paper of Bj\"orner~\cite{bjornersurvey}.  A proof is given for completeness in Corollary~\ref{crosscut} below.

\begin{Thm}[Rota]\label{crosscutstated}
Let $P$ be a finite poset such that any subset of $P$ with a common lower bound
has a meet.  Define a simplicial complex $K$ with vertex set the set $\mathscr
M(P)$ of maximal elements of $P$ and with simplices those subsets of $\mathscr
M(P)$ with a common lower bound.  Then $K$ is homotopy equivalent to the order
complex $\Delta(P)$.
\end{Thm}

Let $\mathscr M(B)$ be the set of maximal elements of $B\setminus \{1\}$ and
let $\Gamma(\mathscr M(B))$ be the \emph{commutation graph} of $\mathscr M(B)$,
that is, the graph whose vertex set is $\mathscr M(B)$ and whose edges are
pairs $(a,b)$ such that $ab=ba$.

\begin{Thm}\label{commutationgraph}
Let $B$ be a finite geometric left regular band.  Then $\Delta(\wh 0,\wh 1)$ is
homotopy equivalent to the clique complex $\Cliq(\Gamma(\mathscr M(B)))$ of the
commutation graph of the set $\mathscr M(B)$ of maximal elements of $B\setminus
\{1\}$.
\end{Thm}
\begin{proof}
In a finite geometric left regular band, a subset $A\subseteq B$ has a lower
bound if and only if the elements of $A$ all mutually commute, in which case
$A$ has a meet, namely the product of all elements of $A$.  The result now
follows from Rota's cross-cut theorem (Theorem~\ref{crosscutstated}), and the
definition of the commutation graph.
\end{proof}

Theorem~\ref{commutationgraph} can be used to give another proof that if a left
regular band $B$ is right hereditary, then $\Bbbk B$ is hereditary.  This proof also
leads to an easy computation of the quiver.

\begin{Thm}\label{treethm}
Let $B$ be a finite left regular band that is right hereditary and let $\Bbbk$ be a
field.  Then $\Bbbk B$ is hereditary.  The quiver $Q(\Bbbk B)$ has vertex set
$\Lambda(B)$.  The number of edges from $X$ to $Y$ is zero unless $X<Y$.  If
$X<Y$, choose $e_Y\in B$ with $\sigma(e_Y) = Y$.  Then the number of edges from $X$ to $Y$ is one less
than the number of children of $e_Y$ with support greater than or equal to $X$.
\end{Thm}
\begin{proof}
If $X<Y$,  then $B[X,Y]$ is also right hereditary and $\mathscr M(B[X,Y])$
consists of the children of $e_Y$ with support greater than or equal to $X$.  Since the Hasse diagram of $B[X,Y]$ is
a tree, no two elements of $\mathscr M(B[X,Y])$ have a common lower
bound and hence $\Gamma(\mathscr M(B[X,Y]))$ has no edges, and so in particular
is its own clique complex.  Thus $\til H^0(\Delta(X,Y))=|\mathscr M(B[X,Y])|-1$
and $\til H^n(\Delta(X,Y))=0$ for all $n\geq 1$ by
Theorem~\ref{commutationgraph}.  The result now follows immediately from
Theorem~\ref{globaldimension}.
\end{proof}

Theorem~\ref{treethm} covers the algebras of nearly all the left regular bands
considered by K.~Brown in~\cite{Brown1}, except the hyperplane semigroups.  It also
covers the interval greedoid left regular bands of Bj\"orner~\cite{bjorner2}.

\begin{Cor}\label{freelrb}
The algebra $\Bbbk F_n$ is hereditary over any field $\Bbbk$.  The quiver $Q(\Bbbk F_n)$ has
vertex set the subsets of $\{1,\ldots,n\}$.  If $X\supsetneq Y$, then there are
$|X\setminus Y|-1$ edges from $X$ to $Y$.  There are no other edges.
\end{Cor}

The quiver of $\Bbbk F_3$ is depicted in Figure \ref{fig:F3quiver}.
\begin{figure}[htb]
\begin{gather*}
\xymatrix@l@M=1pt{
  & \{a,b,c\}\ar@/^7ex/[ddd] \ar@/_7ex/[ddd] \ar@/_1.5ex/[ddl] \ar@/^1.5ex/[rdd]
\ar@<-0.4ex>@/^4ex/[dd] \\
\{a,b\} \ar@/_2ex/[ddr] & \ \{a,c\} \ar@<0.4ex>@/_4ex/[dd] & \{b,c\} \ar@/^2ex/[ddl] \\
\{a\} & \{b\} \ & \{c\} \\
 & \emptyset
}
\end{gather*}
\caption{The quiver of $\Bbbk F(\{a,b,c\})$, cf. Figure \ref{figure: free lrb}.}
\label{fig:F3quiver}
\end{figure}

\begin{proof}
The free left regular band is right hereditary.  If $X\supsetneq Y$ and $w$ is a
word with support $Y$, then the children of $w$ with support greater than or equal to $X$ are the
words $wx$ with $x\in X\setminus Y$.  This completes the proof.
\end{proof}

Let us generalize the above result to Karnofsky-Rhodes expansions.

\begin{Cor}
Let $L$ be a finite lattice with monoid generating set $A$ and let $\Bbbk$ be a
field.  Let $\wh L_A$ be the Karnofsky-Rhodes expansion of $L$ with respect to
$A$.  Then $\Bbbk\wh L_A$ is hereditary.  The quiver of $\Bbbk\wh L_A$ has vertex set
$L$.  The number of edges from $X$ to $Y$ is zero unless $X<Y$, in which case
it is one less than the number of elements $a\in A$ such that $X\leq Y\wedge
a<Y$.
\end{Cor}
\begin{proof}
Again, we have $\wh L_A$ is right hereditary.  If $w$ is a reduced word with
support $Y$, then the children of $w$ with support greater than or equal to $X$ are the words $wa$ such that $Y\wedge
a<Y$ and $Y\wedge a\geq X$.
\end{proof}

One similarly has that the Rhodes expansion of a lattice is right hereditary
and one can explicitly write down its quiver.  The number of edges from $X$ to
$Y$ when $X<Y$ is one less than the number of elements $Z\geq X$, which are
covered by $Y$.

\begin{Rmk}
One can alternatively prove Theorem~\ref{treethm} via a counting argument using
the description above of the quiver and Gabriel's theorem.
\end{Rmk}

\subsection{Free partially commutative left regular bands}
We prove that the global dimension of a free partially commutative left regular
band is the Leray number of the clique complex of the corresponding graph.  This gives a new interpretation of the Leray number of a clique complex in terms of non-commutative algebra.

\begin{Thm}\label{partiallycommmain}
Let $\Gamma=(V,E)$ be a finite graph and $\Bbbk$ a commutative ring with unit.
Then, for $W\subsetneq U\subseteq V$, we have
\begin{align*}
    \Ext^n_{\Bbbk B(\Gamma)}(\Bbbk_U,\Bbbk_W)=\til H^{n-1}(\Cliq(\Gamma[U\setminus W]),\Bbbk)
\end{align*}
for $n\geq 1$.
\end{Thm}
\begin{proof}
One has that $B(\Gamma)[U,W]=B(\Gamma[U\setminus W])$.  Thus we may assume
without loss of generality that $W=\emptyset$ and $U=V$.  The maximal elements
of $B(\Gamma)\setminus \{1\}$ are the elements of $V$.  The commutation graph
for this set is exactly $\Gamma$.  Since $B(\Gamma)$ is a geometric left
regular band, we conclude $\Delta(\wh 0,\wh 1)$ is homotopy equivalent to
$\Cliq(\Gamma)$ by Theorem~\ref{commutationgraph}.  The theorem now follows from Theorem~\ref{mainresult}.
\end{proof}

We present two immediate corollaries.  The first characterizes the free
partially commutative left regular bands with a hereditary $\Bbbk$-algebra.

\begin{Cor}
\label{FPLBgldim}
If $\Bbbk$ is a field and $\Gamma$ a finite graph, then the global dimension of
$\Bbbk B(\Gamma)$ is the $\Bbbk$-Leray number $L_\Bbbk(\Cliq(\Gamma))$. In particular,
$\Bbbk B(\Gamma)$ is hereditary if and only if $\Gamma$ is a chordal graph.
\end{Cor}

Our next corollary computes the quiver of the algebra of a free partially commutative left
regular band.

\begin{Cor}
Let $\Gamma=(V,E)$ be a finite graph.
The quiver of $\Bbbk B(\Gamma)$ has vertex set the power set of $V$. If $U\supsetneq
W$, then the number of edges from $U$ to $W$ is one less than the number of
connected components of $\Gamma[U\setminus W]$.  There are no other edges.
\end{Cor}

Corollary~\ref{freelrb} can be recovered from these results
by specializing to the case that $\Gamma$ has no edges.

It is easy to see that if $\Gamma$ is \emph{triangle-free}, that is, has no
$3$-element cliques, then $\Cliq(\Gamma)=\Gamma$ and so $L_\Bbbk(\Cliq(\Gamma))=2$ unless
$\Gamma$ is a forest (in which case $\Gamma$ is chordal). This provides a natural infinite family of finite
dimensional algebras of global dimension $2$.

It is known that $L_\Bbbk(\Cliq(\Gamma))$
is bounded by the minimal number of chordal graphs needed to cover $\Gamma$~\cite[Theorem 13]{Woodroofe}.  If $C_n$ is the cycle with $n$~nodes and $P_n$ is
the path with $n$~nodes, then it is known that
\begin{align*}
    L_\Bbbk(\Cliq(\ov C_n))=L_\Bbbk(\Cliq(\ov P_n)) = \left\lfloor \frac{n-2}{3}\right\rfloor +1
\end{align*}
for $n\geq 3$~\cite[Proposition 9]{Woodroofe}.  If
$\Gamma_1,\Gamma_2$ are two graphs, then
\begin{align*}
    L_\Bbbk(\Cliq(\Gamma_1\ast \Gamma_2))=L_\Bbbk(\Cliq(\Gamma_1))+L_\Bbbk(\Cliq(\Gamma_2))
\end{align*} where $\ast$ denotes the join of graphs; see~\cite[Lemma 8]{Woodroofe}.
If $\ov \Gamma$ is chordal,
then $L_\Bbbk(\Cliq(\Gamma))$ is the maximal size of an induced matching in $\ov
\Gamma$~\cite[Corollary 18]{Woodroofe}.
If $\Gamma$ is planar, $L_\Bbbk(\Cliq(\Gamma))\leq 3$~\cite[Proposition 23]{Woodroofe}.


\subsection{Hyperplane face monoids}
We return to the setting of Section \ref{ss:hyperplanefacemonoids}.
Recall that $\AAA$ denotes a central hyperplane arrangement in a
$d$-dimensional real vector space, $\LLL$ its intersection lattice, and $\FFF$
its monoid of faces. Without loss of generality, we can suppose that the
intersection of all the hyperplanes in $\AAA$ is the origin: otherwise quotient
the vector space by this intersection; the resulting monoid of faces is
isomorphic to $\FFF$.

We argue that $\Delta(\wh0,\wh1)$ is a $(d-1)$-sphere. Note that the $\R$-order on
$\FFF$ can be described geometrically as $y \leq x$ if and only if $x \subseteq
\overline y$, where $\overline y$ denotes the set-theoretic closure of $y$.
This establishes an order-reversing bijection between the faces $\FFF$ and the
cells of the regular cell decomposition $\varSigma$ obtained by intersecting
the hyperplane arrangement with a sphere centered at the origin. The dual of
$\varSigma$ is the boundary of a polytope $Z$ (a zonotope, actually), and so
the poset of faces of $Z$ is isomorphic to $\FFF$ \cite[Section
2E]{DiaconisBrown1}. Since the order complex of the poset of faces of a
polytope is the barycentric subdivision of the polytope, it follows that
$\Delta(\wh0,\wh1)$ is a $(d-1)$-sphere.

This argument also applies to $\Delta(X,Y)$ for $X \leq Y$ in $\LLL$.
Indeed, $\FFF[X,Y]$ is the face monoid of the hyperplane arrangement in $X$
obtained by intersecting $X$ with the hyperplanes $H \in \AAA$ containing $Y$
but not $X$. It follows that $\Delta(X,Y)$ is a sphere of dimension
$\dim(X)-\dim(Y)-1$. Consequently, we recover Lemma~8.3 of
\cite{Saliolahyperplane}.

\begin{Prop}
\label{ext-spaces}
For $X, Y \in \LLL$ and $n \geq 0$,
\[
\Ext^n_{\Bbbk\FFF}(\Bbbk_X, \Bbbk_Y) \cong
\begin{cases}
\Bbbk, & \text{if } Y \subseteq X \text{ and } \dim(X) - \dim(Y) = n, \\
0, & \text{otherwise}.
\end{cases}
\]
\end{Prop}
\begin{proof}
    We apply Theorem~\ref{mainresult}. Since $\Delta(X,Y)$ is a sphere of
    dimension $\dim(X)-\dim(Y)-1$, it follows that $\til
    H^{n-1}(\Delta(X,Y),\Bbbk)$ is $0$ unless $\dim(X)-\dim(Y)=n$,
    in which case it is $\Bbbk$.
\end{proof}
It follows that the quiver of $\Bbbk\FFF$ coincides with the Hasse diagram
of $\LLL$ ordered by reverse inclusion.
\begin{Cor}[{Saliola~\cite[Corollary~8.4]{Saliolahyperplane}}]
The quiver of $\Bbbk\FFF$ has vertex set $\LLL$.
The number of arrows from $X$ to $Y$ is zero unless
$Y \subsetneq X$ and $\dim(X)-\dim(Y)=1$, in which case
there is exactly one arrow.
\end{Cor}

In \cite{Saliolahyperplane} a set of quiver relations for $\Bbbk\FFF$ was
described: for each interval of length two in $\LLL$ take the sum of all paths
of length two in the interval. It was also shown that $\Bbbk\FFF$ is a Koszul
algebra and that its Koszul dual algebra is isomorphic to the incidence algebra
of the intersection lattice $\LLL$.

\subsection{Complex hyperplane face monoids}
We show that the situation for complex hyperplane arrangements is similar to that of real hyperplane arrangements.  Things are slightly more complicated in this setting because interval submonoids of complex hyperplane face monoids are not again complex hyperplane face monoids.  So we have to exploit much more the PL structure of the cell complex associated to the arrangement.

Fix a complex hyperplane arrangement $\mathcal A$ in $\mathbb C^d$ with augmented intersection lattice $\LLL$.
  The main technical result we shall need is the following.

\begin{Prop}\label{isasphere}
Let $X,Y\in \LLL$ with $X<Y$.  Then $\Delta(X,Y)$ is a sphere of dimension $\dim_{\mathbb R} X-\dim _{\mathbb R}Y-1$.
\end{Prop}

Assuming for the moment the proposition, we obtain the following analogues of the results from the real case.

\begin{Prop}
For $X, Y \in \LLL$ and $n \geq 0$,
\[
\Ext^n_{\Bbbk\FFF}(\Bbbk_X, \Bbbk_Y) \cong
\begin{cases}
\Bbbk, & \text{if } Y \subseteq X \text{ and } \dim_{\mathbb R}(X) - \dim_{\mathbb R}(Y) = n, \\
0, & \text{otherwise}.
\end{cases}
\]
\end{Prop}
\begin{proof}
   This is an application of Theorem~\ref{mainresult}. Since $\Delta(X,Y)$ is a sphere of
    dimension $\dim_{\mathbb R}(X)-\dim_{\mathbb R}(Y)-1$ by Proposition~\ref{isasphere}, it follows that $\til
    H^{n-1}(\Delta(X,Y),\Bbbk)$ is $0$ except for when $\dim_{\mathbb R}(X)-\dim_{\mathbb R}(Y)=n$,
    in which case it is $\Bbbk$.
\end{proof}

An immediate consequence is that the quiver of $\Bbbk\FFF$ coincides with the Hasse diagram
of $\LLL$ ordered by reverse inclusion, as was the case for real hyperplane arrangements.
\begin{Cor}
The quiver of $\Bbbk\FFF$ has vertex set $\LLL$.
The number of arrows from $X$ to $Y$ is zero unless
$Y \subsetneq X$ and $\dim(X)-\dim(Y)=1$, in which case
there is exactly one arrow.
\end{Cor}

We now prove Proposition~\ref{isasphere}.
\begin{proof}[Proof of Proposition~\ref{isasphere}]
Viewing $\FFF^{op}$ as the face poset of a regular CW decomposition of $S^{2d-1}$, one has that $\FFF_{\geq X}^{op}$ is the subcomplex of $\FFF^{op}$ corresponding  to the intersection $X\cap S^{2d-1}$ (see~\cite[Theorem 3.5]{bjorner2} and the discussion preceding it). In~\cite[Theorem~2.6]{BjornerZiegler1992}, it is shown that there is a real hyperplane arrangement $\mathcal A'$ in $\mathbb R^{2d}=\mathbb C^d$ such that the regular CW complex decomposition of $S^{2d-1}$ induced by $\mathcal A'$ is a subdivision of the CW decomposition associated to $\mathcal A$.  Recalling that $X$ is a real subspace of $\mathbb C^d$, it follows that the subcomplex $X\cap S^{2d-1}$ has a subdivision coming from the real hyperplane arrangement $\mathcal A''=\{H\cap X\mid H\in \mathcal A'\ \text{and}\ X\nsubseteq H\}$ in $X$.  Thus $\FFF_{\geq X}^{op}$  is the face poset of a PL regular CW decomposition of $S^{\dim_{\mathbb R} X-1}$, cf.~\cite[Theorem~2.2.2]{OrientedMatroids1993}.   Therefore, the order complex $\Delta(\FFF_{\geq X}^{op})$ is a PL sphere of dimension $\dim_{\mathbb R} X-1$, cf.~\cite[Lemma~4.7.25]{OrientedMatroids1993}.

Fix $F\in \FFF$ with support $Y$.  Let $c$ be a maximal chain in $\FFF_{\geq X}^{op}$ from the identity to $F$.  It follows from Lemma~\ref{Rhodeselem} that the length of $c$ is the same as the length of the interval $[Y,\wh 1]$ of $\LLL$, which is $\dim_{\mathbb R} Y$. Next we observe that $\Delta(\FFF[X,Y)^{op})$ is the link of $c$ in $\Delta(\FFF_{\geq X}^{op})$ and thus is a PL sphere of dimension $\dim_{\mathbb R} X-\dim_{\mathbb R} Y-1$ by~\cite[Theorem~4.7.21(iv)]{OrientedMatroids1993}.  Since the order complex of a poset and its opposite are the same, we have $\Delta(X,Y)=\Delta(\FFF[X,Y))\cong S^{\dim_{\mathbb R} X-\dim_{\mathbb R} Y-1}$.
\end{proof}

\section{Proof of the Main Result: the Algebraic Component}
\label{s:mainresult}

We begin the proof of the main result. In this section, we develop the tools that will allow us to recast the
computation of $\Ext^n_{\Bbbk B}(\Bbbk_X,\Bbbk_Y)$ into one involving monoid cohomology and
classifying spaces. No knowledge of monoid cohomology or classifying spaces is
required to read this section.

Let $\Bbbk$ be a commutative ring with unit and fix a finite left regular band $B$.
We begin by describing certain projective modules for a left regular band algebra. The
most important case is when $\Bbbk$ is a field, but it is conceivable that the case
$\Bbbk=\mathbb Z$ may be of interest in the future and so we do not restrict
ourselves here. We use module unqualified to mean left module.

\subsection{Orthogonal idempotents}
\label{ss:orthogonalidempotents}
Many of the results of the second author~\cite{Saliola} from the case where $\Bbbk$
is a field generalize to any commutative ring with unit.  In particular, he
constructed a complete set of
orthogonal idempotents that are defined over $\mathbb ZB$.  In the case $\Bbbk$ is
a field, they are the primitive idempotents.  But they are useful in general
since they give us a decomposition of $\Bbbk B$ into a direct sum of projective
modules.

Fix for each $X\in \Lambda(B)$ an element $f_X$ with $X=Bf_X$.  Define $e_X$
recursively by $e_{\wh 0} = f_{\wh 0}$ and, for $X>\wh 0$,
\begin{equation}\label{defidempotents}
e_X = f_X\left(1-\sum_{Y<X}e_Y\right).
\end{equation}
Notice that, by induction, one can write
\begin{align*}
    e_X=\sum_{b\in B} c_bb
\end{align*}
with the $c_b$ integers such that $f_X\geq b$ for all $b$ with $c_b\neq 0$ and
the coefficient of $f_X$ in $e_X$ is $1$.

The following results are proved in Lemma~4.1 and Theorem~4.2 of~\cite{Saliola}
when $\Bbbk$ is a field, but the proofs make no use of the assumption that $\Bbbk$ is a field.

\begin{Thm}\label{primidempotentprops}
The elements $\{e_X\}_{X\in \Lambda(B)}$ enjoy the following properties:
\begin{enumerate}
\item if $b\in B$ and $X\in \Lambda(B)$ are such that $b\in \nup X$, then
    $be_X=0$;
\item $\{e_X\}_{X\in \Lambda(B)}$ is a complete set of orthogonal idempotents,
    that is, one has $e_Xe_Y = \delta_{X,Y}e_X$ and
    \begin{align*}
        \sum_{X\in \Lambda(B)} e_X=1;
    \end{align*}
\item if $\Bbbk$ is a field, then $e_X$ is a primitive idempotent.
\end{enumerate}
\end{Thm}

The following is~\cite[Corollary~4.4]{Saliola}, but we must adapt the proof
since we are not assuming $\Bbbk$ is a field.

\begin{Cor}\label{basisofidempotents}
The set $\{be_{\sigma(b)}\mid b\in B\}$ is a basis of idempotents for $\Bbbk B$.
\end{Cor}
\begin{proof}
Since the map $a\mapsto ba$ is a homomorphism $B\to bB$, it follows that left
multiplication by $b$ is a ring homomorphism $\Bbbk B\to \Bbbk bB$ and hence $be_{\sigma(b)}$
is an idempotent.  Consider the module homomorphism $\p\colon \Bbbk B\to \Bbbk B$ given by $b\mapsto
be_{\sigma(b)}$.  By the remarks before Theorem~\ref{primidempotentprops} and
the fact that $bf_{\sigma(b)}=b$, it follows that $be_{\sigma(b)}$ is an
integral linear combination of elements of $bB$ and that the coefficient of $b$
itself is $1$.  Thus if we order $B$ in a way compatible with the $\mathscr
R$-order, then the matrix of $\p$ is unitriangular and hence invertible over
any commutative ring with unit.  This establishes the corollary.
\end{proof}

\subsection{Sch\"utzenberger representations}
\label{ss:Schutz}
Next we recall the classical (left) Sch\"utzenberger representation associated
to an element $X\in \Lambda(B)$.  Let $L_X=\sigma\inv (X)$; it is called an
$\mathscr L$-class in the semigroup theory literature~\cite{CP}.  Define a
$\Bbbk B$-module structure on $\Bbbk L_X$ by putting, for $a\in B$ and $b\in L_X$,
\begin{align*}
    a\cdot b =
    \begin{cases}
        ab & \text{if}\ \sigma(a)\geq X\\
        0 & \text{else.}
    \end{cases}
\end{align*}

It is proved by the second author~\cite{Saliola} that $\Bbbk Be_X\cong \Bbbk L_X$ when
$\Bbbk$ is a field. The argument is easily adapted to the general case.  Namely,
using Corollary~\ref{basisofidempotents}, the proof of~\cite[Lemma
5.1]{Saliola} goes through to show that a $\Bbbk$-basis for $\Bbbk Be_X$ is the set
$\{be_X\mid \sigma(b)=X\}$. It is then shown in~\cite[Proposition 5.2]{Saliola}
that the map $\p\colon \Bbbk L_X\to \Bbbk Be_X$ given by $b\mapsto be_X$ for $b\in L_X$
is a $\Bbbk B$-module homomorphism.  Since it sends a basis to a basis, it is an
isomorphism.  Putting everything together we have the following theorem.

\begin{Thm}\label{Schutz}
Let $B$ be a left regular band and $\Bbbk$ a commutative ring with unit.  Then the
Sch\"utzenberger representations $\Bbbk L_X$ with $X\in \Lambda(X)$ are projective
and
\begin{align*}
    \Bbbk B\cong \bigoplus_{X\in \Lambda(B)} \Bbbk L_X.
\end{align*}
If $\Bbbk$ is a field, then this is the decomposition of $\Bbbk B$ into projective
indecomposables.
\end{Thm}

If $X\in \Lambda(B)$, there is a $\Bbbk$-algebra homomorphism $\rho_X\colon \Bbbk B\to
\Bbbk B_{\geq X}$ given by
\begin{align*}
    \rho_X(b) =
    \begin{cases}
        b & \text{if}\ \sigma(b)\geq X\\
        0 & \text{else.}
    \end{cases}
\end{align*}
This homomorphism allows us to consider any $\Bbbk B_{\geq X}$-module $M$ as a
$\Bbbk B$-module via the action $b \cdot m = \rho_X(b) m$ for all $b \in B$ and $m
\in M$.

The kernel of $\rho_X$ is the ideal $\Bbbk\nup X$.  If $Y\geq X$, then $\Bbbk L_Y$ is a
$\Bbbk B_{\geq X}$-module and is the corresponding projective module for $\Bbbk B_{\geq
X}$.  We thus obtain the following corollary of Theorem~\ref{Schutz}.

\begin{Cor}\label{niceprojectives}
If $X\in \Lambda(B)$, then $\Bbbk B_{\geq X}$ is a projective $\Bbbk B$-module and the
decomposition
\begin{align*}
    \Bbbk B=\Bbbk B_{\geq X}\oplus \bigoplus_{Y\ngeq X}\Bbbk L_Y
\end{align*}
holds. Consequently, any projective $\Bbbk B_{\geq X}$-module is a projective
$\Bbbk B$-module (via $\rho_X$).
\end{Cor}

As a corollary, it follows that we can compute $\Ext$ in either $\Bbbk B$ or
$\Bbbk B_{\geq X}$ for $\Bbbk B_{\geq X}$-modules.

\begin{Cor}\label{computeExt}
Let $X\in \Lambda(B)$ and let $M,N$ be $\Bbbk B_{\geq X}$-modules.  Then
\begin{align*}
    \Ext^n_{\Bbbk B}(M,N)\cong \Ext^n_{\Bbbk B_{\geq X}}(M,N)
\end{align*}
for all $n\geq 0$.
\end{Cor}
\begin{proof}
Corollary~\ref{niceprojectives} implies any projective resolution of $M$ over
$\Bbbk B_{\geq X}$ is also a projective resolution over $\Bbbk B$.  The result is now
immediate.
\end{proof}

\subsection{Computation of $\Ext^n_{\Bbbk B}(\Bbbk_X,\Bbbk_Y)$}
\label{ss:mainresult}

If $X\in \Lambda(B)$, let $\Bbbk_X$ be the trivial $\Bbbk B_{\geq X}$-module.  It is
then a $\Bbbk B$-module via $\rho_X$.  There is a surjective homomorphism
$\varepsilon_X\colon \Bbbk L_X\to \Bbbk_X$ sending each element of $L_X$ to $1$.  If $\Bbbk$
is a field, then the modules $\Bbbk_X$ form a complete set of non-isomorphic simple
$\Bbbk B$-modules and $\varepsilon_X\colon \Bbbk L_X\to \Bbbk_X$ is the projective
cover~\cite{Brown1,Saliola,rrbg}.  In general, the simple $\Bbbk B$-modules are the
modules of the form $(\Bbbk/\mathfrak m)_X$ where $\mathfrak m$ is a maximal ideal
of $\Bbbk$, as follows from the results of~\cite{myirreps}.

We begin with a construction of a projective resolution of the modules $\Bbbk_X$
with $X\in \Lambda(B)$. Recall that these are the simple $\Bbbk B$-modules when $\Bbbk$
is a field.

\begin{Prop}\label{standardbar}
Let $B$ be a finite left regular band and let {$X\in \Lambda(B)$}.
Define a chain complex $C_\bullet(B,X)$ by putting $C_n(B,X)$ to be the free
$\Bbbk B_{\geq X}$-module on the set {$B[X,\wh 1)^n=(B_{\geq X}\setminus \{1\})^n$}.
Denote a basis element by $[s_0|\cdots|s_{n-1}]$.  When $n=0$, the unique basis
element is denoted $[\,]$.  The augmentation $\varepsilon_X\colon C_0(B,X)\to
\Bbbk_X$ sends $[\,]$ to $1$.  Define
\begin{align*}
    d_n\colon C_n(B,X)\to C_{n-1}(B,X),
\end{align*}
for $n\geq 1$, by
\begin{align*}
d_n([s_0|\cdots|s_{n-1}]) =
    &s_0[s_1|\cdots|s_{n-1}]\\
    &+\sum_{i=1}^{n-1}(-1)^{i}[s_0|\cdots|s_{i-1}s_i|\cdots |s_{n-1}] \\
    &+(-1)^n[s_0|\cdots|s_{n-2}]
\end{align*}
Then $C_\bullet(B,X)\to \Bbbk_X$ is a projective resolution of $\Bbbk_X$ as a $\Bbbk B$-module.
\end{Prop}
\begin{proof}
Observe that $C_\bullet(B,X)\to \Bbbk_X$ is the normalized bar resolution
of $\Bbbk_X$ as a $\Bbbk B_{\geq X}$-module (see~\cite[Chapter X]{MacLaneHomology}).  The result now follows from
Corollary~\ref{niceprojectives}.
\end{proof}

As a consequence, we can show that $\Ext^n_{\Bbbk B}(\Bbbk_X,\Bbbk_Y)$ vanishes when $Y\ngeq
X$.

\begin{Prop}\label{noextincomparablecase}
Let $B$ be a finite left regular band and $\Bbbk$ a commutative ring with unit.
Let $X,Y\in \Lambda(B)$ and assume $Y\ngeq X$.  Then \[\Ext^n_{\Bbbk B}(\Bbbk_X,\Bbbk_Y)=0\]
for all $n\geq 0$.
\end{Prop}
\begin{proof}
Let $a\in B$ with $\sigma(a)=Y$.  Then $a$ annihilates $\Bbbk B_{\geq X}$ and acts
as the identity on $\Bbbk_Y$.  Thus $\Hom_{\Bbbk B}(\Bbbk B_{\geq X},\Bbbk_Y)=0$.  The result now
follows by using the resolution in Proposition~\ref{standardbar} to compute
$\Ext^n_{\Bbbk B}(\Bbbk_X,\Bbbk_Y)$.
\end{proof}

We are left studying the case that $X\leq Y$. This can be recast in terms of
monoid cohomology, which we do in the next section.

\section{Proof of the Main Result: Classifying Spaces and Cohomology}
\label{s:MonoidCohomology}

We now turn to classifying spaces and the cohomology of monoids and categories.
The cohomology of monoids, which is a natural generalization of group
cohomology~\cite{Browncohomology}, is both a special case of the cohomology of
augmented algebras~\cite{CartanEilenberg} and of the cohomology of small
categories~\cite{Webb}. Although we are mostly interested in monoid cohomology,
we will also need the cohomology of small categories. The main example of a
small category that we need that is neither a monoid nor a poset is the
semidirect product of a monoid with a set.

The approach we take is inspired by the paper of Nunes~\cite{Nunes}.
In what follows, $M$ will denote a monoid and the set of idempotents of $M$
will be denoted by $E(M)$. Of course, if $M$ is a band, then $M = E(M)$.

\subsection{Cohomology of a small category}
\label{ss:categorycohomology}

Let $\mathscr C$ be a small category and fix a commutative ring $\Bbbk$ with unit.
By a \emph{left $\mathscr C$-module (over $\Bbbk$)} we mean a (covariant) functor $F\colon
\mathscr C\to \module \Bbbk$.  For instance, if we
view a monoid $M$ as a one-object category, then a left $M$-module (over $\Bbbk$) is the
same thing as a left $\Bbbk M$-module.  The category $\module \Bbbk^{\mathscr C}$ of
left $\mathscr C$-modules is well known to be an abelian category with enough projectives
and injectives~\cite{ringoids}. The morphisms between $\mathscr C$-modules $F,G$ are natural
transformations and we write $\Hom_{\Bbbk\mathscr C}(F,G)$ for the morphism set.

There is a functor $\Delta\colon \module \Bbbk\to \module \Bbbk^{\mathscr C}$ that
sends a $\Bbbk$-module $V$ to the constant functor $\Delta(V)\colon \mathscr C\to \module \Bbbk$
that sends all objects to $V$ and all arrows to the identity $1_V$.  For
example, if $M$ is a monoid, then $\Delta(V)$ is $V$ with the trivial
$\Bbbk M$-module structure.  The functor $\Delta$ has a right adjoint
$\ilim$~\cite{Mac-CWM}.  For instance, if $V$ is a $\Bbbk M$-module, then $\ilim V$
is the $\Bbbk$-module of $M$-invariants.  One has a natural isomorphism $\ilim
F\cong\Hom_{\Bbbk\mathscr C}(\Delta(\Bbbk),F)$ for any left $\mathscr C$-module $F$.  Thus the right derived functors $R^n\ilim$ can be identified
with $\Ext^n_{\Bbbk\mathscr C}(\Delta(\Bbbk),-)$ (where the subscript ``$\Bbbk\mathscr C$''
is to indicate we are considering left $\mathscr C$-modules over $\Bbbk$).  One defines the \emph{cohomology} of $\mathscr C$ with coefficients in
$F$ by
\begin{align*}
    H^n(\mathscr C,F) = \Ext^n_{\Bbbk\mathscr C}(\Delta(\Bbbk),F)=R^n\ilim F
\end{align*}
where $R^n$ denotes the $n^{th}$-right derived functor.
When $\mathscr C$ is a monoid or a group, this agrees with usual monoid and
group cohomology.  See~\cite{Webb,Algebraichomotopy,GabrielZisman,Quillen} for more details.

We are primarily interested in monoid cohomology.  In this case, we
will write $\Bbbk$ instead of $\Delta(\Bbbk)$, as is customary. In particular, if $M$ is a monoid and $V$ is a $\Bbbk M$-module, then $H^n(M,V)=\Ext^n_{\Bbbk M}(\Bbbk,V)$.  See~\cite[Chapter X]{MacLaneHomology} for associated bar resolutions.

  There was some work on
monoid cohomology in the late sixties and early
seventies~\cite{AdamsRieffel,Nicocohom1,Nicocohom2}.  Since the nineties, there
has been a surge in papers on monoid homology and cohomology, in part due to
connections with string rewriting
systems~\cite{brownrewriting,Squier,GubaPride,CohenFP,Nunes,GubaPrideBurnside,KobayashiOtto,Pridefiniteness,Lafont,KobayashiFP1,GrayPride}.

The following proposition follows immediately from
Corollary~\ref{computeExt} and the definition of monoid cohomology.

\begin{Prop}\label{cohomologicalinterp}
Let $B$ be a finite left regular band and $\Bbbk$ a commutative ring with unit. Let
$X\in \Lambda(B)$.  There is a natural isomorphism of functors
\begin{align*}
    \Ext^n_{\Bbbk B}(\Bbbk_X,-)\cong H^n(B_{\geq X},-)
\end{align*}
from the category of $\Bbbk B_{\geq X}$-modules to the category of $\Bbbk$-modules.  In
particular, if $Y\geq X$, then
\begin{equation}\label{simplifiedversion1}
\Ext^n_{\Bbbk B}(\Bbbk_X,\Bbbk_Y)\cong H^n(B_{\geq X},\Bbbk_Y)
\end{equation}
for all $n\geq 0$.
\end{Prop}

\subsection{The Eckmann-Shapiro lemma}
In what follows we shall need a variant of the Eckmann-Shapiro lemma for
monoids. It will be convenient to use a very general Eckmann-Shapiro type lemma
in the context of abelian categories that was proved by Adams and Rieffel in
their study of semigroup cohomology~\cite[Theorem~1]{AdamsRieffel}.

\begin{Thm}\label{AdamsRieffel}
Let $\mathscr A,\mathscr B,\mathscr C$ be abelian categories such that
$\mathscr A$ has enough injectives and suppose that one has a diagram of
additive functors
\begin{align*}
    \xymatrix{\mathscr A\ar@/^/[r]^S&\mathscr B\ar@/^/[l]^T\ar[r]^F& \mathscr C}
\end{align*}
such that $S$ is right adjoint to $T$ and $S,T$ are exact. Then there is a
natural isomorphism of functors  $R^n(F\circ S)\cong (R^nF)\circ S$ for all
$n\geq 0$.
\end{Thm}

Let $\p\colon M\to N$ be a semigroup homomorphism of monoids.  That is,
$\p(m_1m_2)=\p(m_1)\p(m_2)$ for all $m_1,m_2\in M$, but we do not assume
$\p(1)=1$.  Let $e=\p(1)$; it is an idempotent.  Then $\Bbbk Ne$ is a
$\Bbbk N$-$\Bbbk M$-bimodule and $\Bbbk eN$ is a $\Bbbk M$-$\Bbbk N$-bimodule.  If $V$ is a $\Bbbk N$-module, we have that $eV$ is a
$\Bbbk M$-module via the action $mv=\p(m)v$.  Notice that $1v=ev=v$ since $v\in eV$.
The functor $V\mapsto eV$ is exact and has left adjoint $W\mapsto
\Bbbk Ne\otimes_{\Bbbk M}W$ and right adjoint $W\mapsto \Hom_M(eN,W)$ where
$\Hom_M(eN,W)$ is the set of all left $M$-set morphisms $f\colon eN\to W$ with
pointwise $\Bbbk$-module structure and $N$-action $nf(n') = f(n'n)$.  Notice that
$\Hom_M(eN,W)\cong \Hom_{\Bbbk M}(\Bbbk eN,W)$ via restriction to the basis $eN$.  Thus
$\Hom_M(eN,-)$ is exact whenever $\Bbbk eN$ is a projective $\Bbbk M$-module.
Theorem~\ref{AdamsRieffel} specializes in this context as follows.

\begin{Lemma}\label{Eckmannshapiro}
Let $\p\colon M\to N$ be a semigroup homomorphism of monoids and let $e=\p(1)$.
Let $\Bbbk$ be a commutative ring with unit and suppose that $\Bbbk eN$ is a projective
$\Bbbk M$-module.  Then, for any $\Bbbk M$-module $W$ and $\Bbbk N$-module $V$, one has a
natural isomorphism
\begin{align*}
    \Ext^n_{\Bbbk M}(eV,W)\cong \Ext^n_{\Bbbk N}(V,\Hom_M(eN,W))
\end{align*}
for all $n\geq 0$.  In particular,
\begin{align*}
    H^n(M,W)\cong H^n(N,\Hom_M(eN,W))
\end{align*}
for all $n\geq 0$.
\end{Lemma}
\begin{proof}
In Theorem~\ref{AdamsRieffel}, take $\mathscr A=\module {\Bbbk M}$, $\mathscr B=\module {\Bbbk N}$ and $\mathscr C=\module \Bbbk$, take $T(V)=eV$, $S(W) =
\Hom_M(eN,W)$ and $F=\Hom_{\Bbbk N}(V,-)$ and use that $F\circ S =
\Hom_{\Bbbk N}(V,\Hom_M(eN,-))\cong \Hom_{\Bbbk M}(eV,-)$.  The final statement follows because $e\Bbbk=\Bbbk$.
\end{proof}

A corollary we shall use later for a dimension-shifting argument is the
following.  Let $e\in E(M)$ be an idempotent and suppose that $V$ is any
$\Bbbk$-module. Then $V^{eM}$ with the natural $\Bbbk M$-module structure given by
$mf(m') = f(m'm)$ is acyclic for cohomology.

\begin{Cor}\label{niceacyclic}
    Let $V$ be a $\Bbbk$-module and $e\in E(M)$. Then
    \begin{align*}
        H^n(M,V^{eM}) \cong
        \begin{cases}
            V & \text{if}\ n=0\\
            0 &\text{else.}
        \end{cases}
    \end{align*}
\end{Cor}
\begin{proof}
    Let $\{1\}$ denote the trivial monoid and consider the semigroup homomorphism
    $\p\colon \{1\}\to M$ with $\p(1)=e$.  The algebra of the trivial monoid is $\Bbbk$
    and $\Bbbk eM$ is a free $\Bbbk$-module.  Thus Lemma~\ref{Eckmannshapiro} applies.
    Observe that $\Hom_1(eM,V)=V^{eM}$ and so $H^n(M,V^{eM})\cong H^n(\{1\},V)$.
    The result follows.
\end{proof}

\subsection{Classifying space of a small category}
\label{ss:classifyingspace}
To each small category $\mathscr C$, there is naturally associated a CW complex $\mathcal B\mathscr C$ called the \emph{classifying space of $\mathscr C$}.  Usually, $\mathcal B\mathscr C$ is defined as the geometric realization of a certain simplicial set called the nerve of $\mathscr C$~\cite{ktheory}, but we follow instead the construction in~\cite[Definition~5.3.15]{Rosenberg}.

There is a $0$-cell of $\mathcal B\mathscr C$ for each object of the category $\mathscr C$.  For $q\geq 1$, there is a $q$-cell for each diagram
\begin{equation}\label{simplex}
\xymatrix{c_0\ar[r]^{f_0}& c_1\ar[r]^{f_1}&\cdots\ar[r]^{f_{q-2}}&c_{q-1}\ar[r]^{f_{q-1}}&c_q}
\end{equation}
with no $f_i$ an identity arrow.  This $q$-cell (which should be thought of as a $q$-simplex) is attached in the obvious way to any cell of smaller dimension that can be obtained by deleting some $c_i$ and, if $i\notin \{0,q\}$, replacing $f_{i-1}$ and $f_{i}$ by $f_{i}f_{i-1}$ (where if this is an identity morphism then we delete this arrow as well). Notice that $\mathcal B\mathscr C$ is homeomorphic to $\mathcal B\mathscr C^{op}$~\cite{Quillen}.

A functor $F\colon
\mathscr C\to \mathscr D$ induces a cellular map $\mathcal BF\colon \mathcal B\mathscr C\to \mathcal B\mathscr D$
by applying $F$ to \eqref{simplex} and deleting identity morphisms.
Thus $\mathcal B$ is a functor from the category of small categories to the category of CW complexes.

Let us give some examples.  If $G$ is a group, viewed as a one-object category,
then $\mathcal BG$ is Milnor's classifying space of $G$ and is an Eilenberg-Mac~Lane
$K(G,1)$-space.  When $M$ is a monoid, viewed as a one-object category,
one has that $\mathcal BM$ is the classical classifying space of $M$. It is
known that every CW complex is homotopy equivalent to the classifying space of
a monoid~\cite{mcduff}.  It is easy to see that if $V$ is a $\Bbbk$-module, viewed
as a $\Bbbk M$-module via the trivial action, then the co-chain complex associated
to the cellular cohomology of $\mathcal BM$ with coefficients in $V$ is precisely the co-chain complex used to compute
$H^n(M,V)$ if one uses the standard bar resolution~\cite[Chapter X]{MacLaneHomology} for the trivial $\Bbbk M$-module.
Thus $H^\bullet(M,V)=H^\bullet(\mathcal BM,V)$.  More generally, one has, the following theorem, see for instance~\cite{Quillen}, \cite[Theorem 5.3]{Webb}
or~\cite[Appendix II, 3.3]{GabrielZisman}.

\begin{Thm}
For a small category $\mathscr C$, one has $H^\bullet(\mathscr C,\Delta(V))\cong
H^\bullet(\mathcal B\mathscr C,V)$ for any $\Bbbk$-module $V$.
\end{Thm}

Next we recall that a poset $P$ can be viewed as a category.  The object set of
$P$ is $P$ itself.  The arrow set is $\{(p,q)\in P\times P\mid  p\leq q\}$.
The arrow $(p,q)$ goes from $p$ to $q$ and composition is given by
$(q,r)(p,q)=(p,r)$.  The identity at $p$ is $(p,p)$.  Functors between posets
are precisely order preserving maps.  Galois connections are adjunctions
between posets.  Notice the $q$-cells of $\mathcal BP$ are
precisely the chains of length $q$ in $P$ and the gluing is by homeomorphisms of the faces.  Thus $\mathcal BP$ is the \emph{order complex} $\Delta(P)$; see~\cite{Quillen}.  

A key result of Segal~\cite{GSegal} is the following (cf.~\cite[Lemma~5.3.17]{Rosenberg}).

\begin{Lemma}[Segal]\label{SegalLemma}
If $F,G\colon \mathscr C\to \mathscr D$ are functors between small categories and there is a natural transformation $F\Rightarrow G$, then $\mathcal BF$ and $\mathcal BG$ are homotopic.
\end{Lemma}

In particular, if $\mathscr C$ and $\mathscr D$ are naturally equivalent, then
$\mathcal B\mathscr C$ and $\mathcal B\mathscr D$ are homotopy equivalent.  More
generally, if a functor $F\colon \mathscr C\to \mathscr D$ has an adjoint, then $\mathcal B\mathscr C$ and $\mathcal B\mathscr D$ are homotopy equivalent~\cite[Corollary 3.7]{ktheory}.  Actually, we have the following more general consequence.

\begin{Prop}\label{homotopyequiv}
Let $F\colon \mathscr C\to \mathscr D$ and $G\colon \mathscr D\to \mathscr C$
be functors such that there exist natural transformations between $GF$ and
$1_{\mathscr C}$ and between $FG$ and $1_{\mathscr D}$, in either direction.
Then $\mathcal B\mathscr C$ is homotopy equivalent to $\mathcal B\mathscr D$.
\end{Prop}

An important special case is the well-known fact that a Galois connection
between posets yields a homotopy equivalence of order complexes~\cite{bjornersurvey}. The following
corollary is the special case of Rota's cross-cut theorem~\cite{bjornersurvey} that we used earlier (Theorem~\ref{crosscutstated}).

\begin{Cor}[Rota]\label{crosscut}
Let $P$ be a finite poset such that any subset of $P$ with a common lower bound
has a meet.  Define a simplicial complex $K$ with vertex set the set $\mathscr
M(P)$ of maximal elements of $P$ and with simplices those subsets of $\mathscr
M(P)$ with a common lower bound.  Then $K$ is homotopy equivalent to the order
complex $\Delta(P)$.
\end{Cor}
\begin{proof}
Let $\mathscr F$ be the face poset of $K$.  Then it is well known that
$\Delta(\mathscr F)$ is the barycentric subdivision of $K$ and hence
homeomorphic to it.  Thus it suffices, by Proposition~\ref{homotopyequiv}, to
establish a contravariant Galois connection between $P$ and $\mathscr F$.
Define $F\colon P\to \mathscr F$ by $F(p) = \{m\in \mathscr M(P)\mid m\geq p\}$
and $G\colon \mathscr F\to P$ by $G(X)=\bigwedge X$ for a simplex $X$ of $K$.
Then $X\leq F(p)$ if and only if $p\leq G(X)$ and so $F$ and $G$ form a
contravariant Galois connection.
\end{proof}

Another well-known  corollary is that a category with a terminal object has a
contractible classifying space~\cite[Corollary 3.7]{ktheory}.

\begin{Cor}\label{terminal}
If $\mathscr C$ is a category with a terminal object, then $\mathcal B\mathscr C$ is contractible.
\end{Cor}
\begin{proof}
Let $\{1\}$ denote the trivial monoid and let $t$ be the terminal object of
$\mathscr C$.  Then the unique functor $F\colon \mathscr C\to \{1\}$ and the
functor $G\colon \{1\}\to \mathscr C$ sending the unique object of $\{1\}$ to $t$ form
an adjoint pair.
\end{proof}

Let us prove the folklore result that a monoid with a left zero element has a
contractible classifying space. Our proof is easier than the one in~\cite{mcduff}.

\begin{Prop}[Folklore]\label{leftzero}
Let $M$ be a monoid with a left zero element.  Then $\mathcal BM$ is
contractible.
\end{Prop}
\begin{proof}
Let $\p\colon M\to \{1\}$ be the trivial homomorphism and $\psi\colon \{1\}\to M$
be the inclusion.  Trivially, $\p\psi=1_{\{1\}}$.  Next we define a natural
transformation $\eta\colon 1_M\Rightarrow \psi\p$. Let $z$ be the left zero.
The component of $\eta$ at the unique object of $M$ is $z$.  Then one has, for
$m\in M$, that $zm=z=\psi\p(m)z$, i.e., $\eta$ is
a natural transformation.  Thus $\mathcal BM$ is contractible.
\end{proof}

\subsection{Quillen's Theorem A}
\label{ss:quillentheoremA}

An important tool for determining whether a functor $F\colon \mathscr C\to
\mathscr D$ induces a homotopy equivalence of classifying spaces is the famous
Quillen's Theorem A~\cite{Quillen,ktheory}.  If $d$ is an object of $\mathscr
D$, then the \emph{left fiber category} $F/d$ has object set all morphisms
$f\colon F(c)\to d$ with $c$ an object of $\mathscr C$.  A morphism
from $f\colon F(c)\to d$ to $f'\colon F(c')\to d$ in
$F/d$ is a morphism $g\colon c\to c'$ such that
\begin{align*}
    \xymatrix{F(c)\ar[rr]^{F(g)}\ar[rd]_{f}&&F(c')\ar[ld]^{f'}\\ & d & }
\end{align*}
commutes.

\begin{Rmk}\label{preordercase}
If $F\colon \mathscr C\to \mathscr D$ is a functor with $\mathscr D$ a poset,
then there is at most one arrow $F(x)\to d$.  Thus $F/d$ can be identified with
the full subcategory of $\mathscr C$ with object set $F\inv(D_{\leq d})$
where $D_{\leq d}$ consists of all objects of $\mathscr D$ less than or
equal to $d$ in the order.
\end{Rmk}

We now state Quillen's Theorem A; see~\cite{Quillen,ktheory} for a proof.

\begin{Thm}[Quillen's Theorem A]
Let $F\colon \mathscr C\to \mathscr D$ be a functor such that the left fiber
categories $F/d$ have  contractible classifying spaces for all objects $d$ of
$\mathscr D$.  Then $F$ induces a homotopy equivalence of classifying spaces.
\end{Thm}

\subsection{Semidirect product of a monoid with a set}

The main example of a small category which is not a monoid or a poset that we
shall need is the semidirect product of a monoid with a set.  Let $M$ be a
monoid and let $X$ be a right $M$-set.  The \emph{semidirect product} (also
known as the \emph{category of elements}, or the \emph{Grothendieck
construction}) $X\rtimes M$ has object set $X$ and arrow set $X\times M$.  An
arrow $(x,m)$ has domain $xm$ and range $x$, and we draw it
\begin{align*}
    (x,m)=\xymatrix{xm\ar[r]^m&x}.
\end{align*}
Composition is given by $(x,m)(xm,n) = (x,mn)$, or in pictures
\begin{align*}
    \xymatrix{xmn\ar[r]^n\ar@/_1.5pc/[rr]_{mn}&xm\ar[r]^m& x}.
\end{align*}
The identity at $x$ is $(x,1)$.  The assignment $X\mapsto X\rtimes M$ is a
functor from the category of right $M$-sets to the category of small
categories.

There is an exact pair of adjoint functors between the categories $\module
{\Bbbk M}$ and $\module \Bbbk^{X\rtimes M}$  to which we can apply the Adams-Rieffel
theorem.   Namely, if $V$ is a $\Bbbk M$-module, we can define a left $X\rtimes M$-module $P_V$ on objects by
$P_V(x)=V$ for all $x\in X$.  We define $P_V(x,m)\colon V\to V$ on a morphism $(x,m)$ by
$v\mapsto mv$.  The functor $P\colon \module {\Bbbk M}\to \module \Bbbk^{X\rtimes M}$ given by $P(V)=P_V$ on objects (with the obvious effect on
morphisms) is clearly exact.  It has right adjoint $G$ which sends a left $X \rtimes M$-module
$Q$ to the direct product $\prod_{x\in X}Q(x)$ with action given by
$(mf)(x)=Q(x,m)f(xm)$ where we view an element of the direct product as a mapping
$f\colon X\to \coprod_{x\in X}Q(x)$ with $f(x)\in Q(x)$ for all $x\in X$.
Plainly $G$ is an exact functor.  The isomorphism
\begin{align*}
    \Hom_{\Bbbk(X\rtimes M)}(P(V),Q)\to \Hom_{\Bbbk M}(V,G(Q))
\end{align*}
sends a natural transformation $\eta\colon
P_V\Rightarrow Q$ to the mapping $\prod_{x\in X}\eta_x$ where $\eta_x\colon V\to Q(x)$
is the component of $\eta$ at the object $x$.  Conversely, any homomorphism
$\eta\colon V\to G(Q)$ gives rise to a natural transformation $P_V\Rightarrow Q$ whose
component at $x$ is the composition of $\eta$ with the projection to the factor
$Q(x)$.

\begin{Thm}\label{Eckmannshapirodiscretefibration}
Let $M$ be a monoid and $X$ a right $M$-set.  Then one has a natural
isomorphism
\begin{align*}
    \Ext_{\Bbbk(X\rtimes M)}^\bullet(P(V),Q)\cong \Ext_{\Bbbk M}^{\bullet}(V,G(Q))
\end{align*}
where $P$ and $G$ are the adjoint functors
considered above.  In particular, one has
\begin{align*}
    H^\bullet(X\rtimes M,Q)\cong H^\bullet(M,G(Q)).
\end{align*}
\end{Thm}
\begin{proof}
Apply Theorem~\ref{AdamsRieffel} taking $\mathscr A=\module \Bbbk^{X\rtimes M}$,
$\mathscr B=\module {\Bbbk M}$, $\mathscr C=\module \Bbbk$, $S=G$, $T=P$ and
$F=\Hom_{\Bbbk M}(V,-)$.  One uses that $F\circ S = \Hom_{\Bbbk M}(V,G(-))\cong
\Hom_{\Bbbk(X\rtimes M)}(P(V),-)$.  The final statement follows because
$P(\Bbbk)=\Delta(\Bbbk)$.
\end{proof}

Let $V$ be a $\Bbbk$-module.  Then one computes readily that $G(\Delta(V))=V^X$
with the $\Bbbk M$-module structure given via the left action $mf(x)=f(xm)$.  Thus
we have the following corollary to
Theorem~\ref{Eckmannshapirodiscretefibration}.  It is the dual of a result of
Nunes proved for the case of homology~\cite{Nunes}; see also the appendix of~\cite{Loday}.  Our proof though is more
conceptual.

\begin{Cor}\label{semidirectcohom}
Let $M$ be a monoid and $\Bbbk$ a commutative ring with unit.  Suppose that $X$ is
a right $M$-set and $V$ is a $\Bbbk$-module.  Then
\begin{align*}
    H^n(M,V^X)\cong H^n(X\rtimes M,\Delta(V))\cong H^n(\mathcal B(X\rtimes M),V)
\end{align*}
for $n\geq 0$.
\end{Cor}
%
%
%

From now on, we say that a right $M$-set $X$ is \emph{contractible} if $\mathcal B(X\rtimes M)$ is contractible.
Let $X$ be a right $M$-set and denote by $\Omega(X)$ the poset of all cyclic
$M$-subsets $xM$ with $x\in X$ ordered by inclusion.  There is a natural
functor $\Phi_X\colon X\rtimes M\to \Omega(X)$ given by $x\mapsto xM$ on
objects and by sending the arrow $(x,m)\colon xm\to x$ to the unique arrow
$xmM\to xM$.  The following result is a key tool in computing the global
dimension of left regular band algebras.

\begin{Thm}\label{equaltoposet}
Suppose that $X$ is a right $M$-set such that each cyclic $M$-subset $xM$ with
$x\in X$ is contractible.  Then $\Phi_X\colon X\rtimes M\to \Omega(X)$ induces
a homotopy equivalence of classifying spaces.
\end{Thm}
\begin{proof}
By Quillen's Theorem A it suffices to show $\mathcal B(\Phi_X/xM)$ is contractible for each
$x\in X$.  But by Remark~\ref{preordercase}, one has that $\Phi_X/xM$ is the
full subcategory of $X\rtimes M$ on those objects $y\in X$ with $yM\subseteq
xM$.  But this subcategory is precisely $xM\rtimes M$, which has contractible classifying space by
assumption.  Thus $\Phi_X$ is a homotopy equivalence.
\end{proof}

The key example of a monoid $M$ and a right $M$-set to which the theorem will be applied is a right ideal $X$ of a von Neumann
regular (or more generally, right $PP$, defined below) monoid $M$.

\begin{Prop}\label{freeMset}
The right $M$-set $M$ is contractible.
\end{Prop}
\begin{proof}
Notice that $1$ is a terminal object of $M\rtimes M$ since $(1,m)\colon m\to 1$
is the unique arrow from $m$ to $1$. Therefore, $\mathcal B(M\rtimes M)$ is contractible by
Corollary~\ref{terminal}.
\end{proof}

Let $\pv{Set}^{M^{op}}$ be the category of right $M$-sets.  This is a special
case of a presheaf category and so the statements below can be obtained from
the more general statements in this context that one can find in a standard
text in category theory, e.g.,~\cite{Borceux1}. The epimorphisms in
$\pv{Set}^{M^{op}}$ are precisely the surjective maps.  Projective objects in
this category are defined in the usual way.  An $M$-set is
\emph{indecomposable} if it cannot be expressed as a coproduct (equals disjoint
union) of two $M$-sets.  Up to isomorphism the projective
indecomposable $M$-sets are those of the form $eM$ with $e$ an idempotent of
$M$.  These are also the cyclic projective $M$-sets up to isomorphism.

\begin{Prop}\label{projMSet}
A projective indecomposable right $M$-set is contractible.
\end{Prop}
\begin{proof}
Let $P$ be a projective indecomposable right $M$-set.  Then $P$ is cyclic,
being isomorphic to $eM$ for some idempotent $e$, and so there is an
epimorphism $\p\colon M\to P$.  Then $\p$ splits and so $P$ is a retract of
$M$.  Hence, by functoriality, $P\rtimes M$ is a retract of $M\rtimes M$ and so
$\mathcal B(P\rtimes M)$ is a retract of $\mathcal B(M\rtimes M)$.  Since
retracts of contractible spaces are contractible, we are done by
Proposition~\ref{freeMset}.
\end{proof}

\subsection{Application to right $PP$ monoids}
A monoid $M$ is called a \emph{right $PP$ monoid} if each principal right ideal
$mM$ is projective in $\pv{Set}^{M^{op}}$. For example, every band is a right
$PP$ monoid, as is every von Neumann regular monoid.  Right $PP$ monoids were first characterized in~\cite{hereditaryactspaper}.  The following description of
right $PP$ monoids is due to Fountain~\cite{fountain}.  Two elements
$m,n$ of $M$ are said to be \emph{$\mathscr L^*$-equivalent} if $mx=my$ if and only if
$nx=ny$ for all $x,y\in M$.  Note that $m\mathrel{\mathscr L^*} n$ if and only
if there is an isomorphism $mM$ to $nM$ taking $m$ to $n$. A monoid is right $PP$ if and only if each $\mathscr L^*$-class of $M$
contains an idempotent~\cite{fountain,hereditaryactspaper}.

Recall that $\R$ denotes the relation associated to Green's $\R$-preorder on
$M$ (see Section \ref{ss:Rorder}). If $R$ is a right ideal of $M$, then the
poset $R/{\R}$ can be identified with the poset of principal right
ideals of $M$ contained in $R$.

\begin{Cor}\label{rightPPcase}
Let $M$ be a right $PP$ monoid, for example a left regular band, and let $R$ be
a right ideal of $M$.  Then $\mathcal B(R\rtimes M)$ is homotopy equivalent to
$\Delta(R/{\R})$.
\end{Cor}
\begin{proof}
In this case, $\Omega(R)=R/{\R}$.  Since $M$ is right $PP$, each
principal ideal $mM$ is projective and hence contractible by
Proposition~\ref{projMSet}.  Theorem~\ref{equaltoposet} now provides the
desired result.
\end{proof}

We continue to denote by $\Bbbk$ a commutative ring with unit. The following
results are related to results of Nunes~\cite{Nunes}, proved using spectral
sequences, in the context of monoid homology.  Let $M$ be a monoid and $e\in
E(M)$ an idempotent.   Suppose that $\emptyset\neq R\subsetneq eM$ is a right
ideal.  Let $V$ be a $\Bbbk$-module.  The exact sequence of right $\Bbbk M$-modules
\begin{align*}
    \xymatrix{0\ar[r]& \Bbbk R\ar[r] & \Bbbk eM\ar[r] & \Bbbk eM/\Bbbk R\ar[r]& 0}
\end{align*}
gives rise to an exact sequence of left $\Bbbk M$-modules
\begin{equation}\label{alongexact}
\xymatrix{0\ar[r]& W\ar[r] & V^{eM}\ar[r]& V^R\ar[r] & 0}\end{equation}
where $W=\{f\in V^{eM}\mid f(R)=0\}$ and the map $V^{eM}\to V^R$ is given by
restriction.

\begin{Thm}\label{dimensionshift}
Let $M$ be a monoid and $e\in E(M)$.  Let $\emptyset\neq R\subsetneq eM$ be a
right ideal and $V$ a $\Bbbk$-module.  Let $W=\{f\in V^{eM}\mid f(R)=0\}$.   Then
\begin{align*}
    H^{n+1}(M,W)\cong\til H^n(\mathcal B(R\rtimes M),V)
\end{align*}
for all $n\geq 0$. Moreover, if $M$ is a right $PP$ monoid then
\begin{align*}
    H^{n+1}(M,W)\cong \til H^n(\Delta(R/{\R}),V)
\end{align*}
where $R/{\R}$ is the poset of
principal right ideals contained in $R$.
\end{Thm}
\begin{proof}
Corollary~\ref{niceacyclic} implies that $H^n(M,V^{eM})=0$ for $n\geq 1$.  Thus
the long exact sequence for cohomology applied to \eqref{alongexact} yields a
short exact sequence
\begin{align*}
    \xymatrix{0\ar[r]&H^n(M,V^R)\ar[r]&H^{n+1}(M,W)\ar[r]& 0}
\end{align*}
for each $n\geq 1$.  In light of Corollary~\ref{semidirectcohom}, this proves
the first statement of the theorem for $n\geq 1$.  Let us examine the initial
terms of the long exact sequence.  Note that $H^0(M,W)=0$ because if $f\in W$
is fixed by $M$ and $r\in R, m\in M$, then $f(em)=mf(e)=mrf(e)=f(erm)=0$. Also
$H^1(M,V^{eM})=0$ so we have a short exact sequence
\begin{align*}
    \xymatrix{0\ar[r]&H^0(M,V^{eM})\ar[r]^{\pi_*} & H^0(M,V^R)\ar[r]&H^1(M,W)\ar[r]& 0} .
\end{align*}
If $X$ is a right $M$-set, the isomorphism $H^0(M,V^X)\cong H^0(X\rtimes
M,\Delta(V))\cong H^0(\mathcal B(X\rtimes M),V)$ in
Corollary~\ref{semidirectcohom} allows us to identify $H^0(M,V^X)$ with those
functions $f\colon X\to V$ which are constant on connected components of
$\mathcal B(X\rtimes M)$.  Since $\mathcal B(eM\rtimes M)$ is contractible,
hence connected, $H^0(M,V^{eM})$ consists of the constant functions $eM\to V$
and so the image of $\pi_*$ is the set of functions which are constant on $R$.
Thus $\mathop{\mathrm{coker}} \pi_*=\til H^0(\mathcal B(R\rtimes M),V)$.  This
completes the proof of the first statement.  The final statement follows from Corollary~\ref{rightPPcase}.
\end{proof}

Suppose now that $B$ is a finite left regular band and let $\wh 0\neq Y\in \Lambda(B)$.  Suppose that $Y=Be$. Then $R=eB\setminus \{e\}$ is a right ideal
and
\begin{align*}
    \Bbbk_Y\cong \{f\in \Bbbk^{eB}\mid f(R)=0\}
\end{align*}
via the map $f\mapsto f(e)$. Applying Theorem~\ref{dimensionshift} we obtain
the following corollary.

\begin{Cor}\label{leftregularbandcomplex}
Let $B$ be a left regular band and let $\wh 0\neq Y\in \Lambda(B)$.  Then
$H^0(B,\Bbbk_Y)=0$ and
\begin{align*}
    H^{n+1}(B,\Bbbk_Y)\cong \til H^n(\Delta(B[\wh 0,Y)),\Bbbk)
\end{align*}
for all $n\geq 0$ where $\Delta(B[\wh 0,Y))$ is the order complex of $B[\wh
0,Y)$.
\end{Cor}

\subsection{The proof of Theorem~\ref{mainresult}}
At this point, we can give the proof of our main result, Theorem~\ref{mainresult}.  Let $B$ be a finite left regular band. Suppose that $X<Y$ in $\Lambda(B)$ and let
$\Delta(X,Y)$ be the order complex of $B[X,Y)$.

\newtheorem*{MainResult}{Theorem~\ref{mainresult}}
\begin{MainResult}
Let $B$ be a finite left regular band and let $\Bbbk$ be a commutative ring with
unit.  Let $X,Y\in \Lambda(B)$.  Then
\begin{align*}
    \Ext^n_{\Bbbk B}(\Bbbk_X,\Bbbk_Y) =
    \begin{cases}
        \til H^{n-1}(\Delta(X,Y),\Bbbk) & \text{if}\ X<Y,\ n\geq 1\\
        \Bbbk&\text{if}\ X=Y,\ n=0\\
        0 & \text{else.}
    \end{cases}
\end{align*}
\end{MainResult}
\begin{proof}
If $Y\ngeq X$, the result is Proposition~\ref{noextincomparablecase}.  For
$X\leq Y$, we have $\Ext^n_{\Bbbk B}(\Bbbk_X,\Bbbk_Y) = H^n(B_{\geq X},\Bbbk_Y)$ by
Proposition~\ref{cohomologicalinterp}.  Next suppose that $X=Y$. Then
$H^n(B_{\geq X},\Bbbk_X) = H^n(\mathcal BB_{\geq X},\Bbbk)$.
Any element of $B_{\geq X}$ with support $X$ is a left zero and hence $\mathcal BB_{\geq X}$ is contractible by Proposition~\ref{leftzero}.  Thus we are left
with the case $X<Y$. But this is handled by
Corollary~\ref{leftregularbandcomplex}.
\end{proof}


\def\malce{\mathbin{\hbox{$\bigcirc$\rlap{\kern-7.75pt\raise0,50pt\hbox{${\tt
  m}$}}}}}\def\cprime{$'$} \def\cprime{$'$} \def\cprime{$'$} \def\cprime{$'$}
  \def\cprime{$'$} \def\cprime{$'$} \def\cprime{$'$} \def\cprime{$'$}
  \def\cprime{$'$}

\end{document}